\documentclass[11pt]{amsart}
\usepackage{amsthm,amsmath,amssymb,latexsym}
\usepackage{hyperref}
\usepackage{xcolor}
\begin{document}
\newtheorem{theorem}{Theorem}[section]
\newtheorem{definition}[theorem]{Definition}
\newtheorem{proposition}[theorem]{Proposition}
\newtheorem{lemma}[theorem]{Lemma}
\newtheorem{remark}[theorem]{Remark}
\newtheorem{corollary}[theorem]{Corollary}
\newtheorem{question}{Question}
\newtheorem{example}{Examples}[section]
\newtheorem{notation}[theorem]{Notation}
\newtheorem{claim}[theorem]{Claim}
\newtheorem{fact}[theorem]{Fact}
\newcommand\cl{\begin{claim}}
\newcommand\ecl{\end{claim}}
\newcommand\rem{\begin{remark}\upshape}
\newcommand\erem{\end{remark}}
\newcommand\ex{\begin{example}\upshape}
\newcommand\eex{\end{example}}
\newcommand\nota{\begin{notation}\upshape}
\newcommand\enota{\end{notation}}
\newcommand\dfn{\begin{definition}\upshape}
\newcommand\edfn{\end{definition}}
\newcommand\cor{\begin{corollary}}
\newcommand\ecor{\end{corollary}}
\newcommand\thm{\begin{theorem}}
\newcommand\ethm{\end{theorem}}
\newcommand\prop{\begin{proposition}}
\newcommand\eprop{\end{proposition}}
\newcommand\lem{\begin{lemma}}
\newcommand\elem{\end{lemma}}
\newcommand\fct{\begin{fact}}
\newcommand\efct{\end{fact}}
\providecommand\qed{\hfill$\quad\Box$}
\newcommand\pr{{Proof:\;}}
\newcommand\prc{\par\noindent{\em Proof of Claim: }}
\newcommand\dom{{\mathrm{dom}}}
\newcommand\Pn{{P_n}}
\newcommand\ord{{ord}}
\newcommand\rk{{rk}}
\newcommand\G{{G^{\n}}}
\newcommand\car{{\rm{char}}}
\newcommand\M{{{\mathcal M}}}
\newcommand\N{{{\mathcal N}}}
\newcommand\K{{{\mathcal K}}}
\newcommand\D{{{\mathcal D}}}
\newcommand\A{{{\mathcal A}}}
\newcommand\Ps{{{\mathcal P}}}
\newcommand\Fs{{{\mathcal F}}}
\newcommand\E{{\mathcal E}}
\newcommand\X{{{\bold{X}}}}
\newcommand\x{{{\bold{x}}}}
\newcommand\y{{{\bold{y}}}}
\newcommand\s{{{\bold{s}}}}
\newcommand\bd{{{\bold{d}}}}
\newcommand\bu{{{\bold{u}}}}
\newcommand\bc{{{\bold{c}}}}
\newcommand\ba{{{\bold{a}}}}
\newcommand\be{{{\bold{e}}}}
\newcommand\bb{{{\bold{b}}}}
\newcommand\z{{{\bold{z}}}}
\newcommand\w{{{\bold{w}}}}
\newcommand\bv{{{\bold{v}}}}
\newcommand\ta{{{\bold{t}}}}
\newcommand{\ra}{{{\bold {r}}}}
\newcommand\Y{{{\bold{Y}}}}
\renewcommand\O{{{\mathcal O}}}
\renewcommand\L{{{\mathcal L}}}
\newcommand\Th{{\text{Th}}}
\newcommand\IZ{{\mathbb Z}}
\newcommand\IQ{{\mathbb Q}}
\newcommand\IR{{\mathbb R}}
\newcommand\IN{{\mathbb N}}
\newcommand\IC{{\mathbb C}}
\newcommand\F{{\mathbb F}}
\newcommand\Se{{\mathcal S}}
\newcommand\V{{\mathcal V}}
\newcommand\W{{\mathcal W}}
\newcommand\T{{\mathcal T}}
\newcommand\si{{\sigma}}
\newcommand\n{{\nabla}}
\newcommand{\RCF}{\mathrm{RCF}}
\newcommand{\ACF}{\mathrm{ACF}}
\newcommand{\DCF}{\mathrm{DCF}}
\newcommand{\ACVF}{\mathrm{ACVF}}
\newcommand{\DLg}{\mathrm{DLg}}
\newcommand{\DL}{\mathrm{DL}}
\newcommand{\divi}{\mathrm{div}}
\newcommand{\td}{\mathrm{td}}
\newcommand{\PCF}{p\mathrm{CF}}
\newcommand{\CODF}{\mathrm{CODF}}
\newcommand{\Ecl}{\mathrm{ecl}}
\newcommand{\dcl}{\mathrm{dcl}}
\newcommand{\acl}{\mathrm{acl}}
\newcommand{\Cl}{\mathrm{cl}}
\newcommand{\IFTA}{\mathrm{IFT}}
\newcommand{\LFF}{\mathrm{LFF}}
\newcommand\C{{\mathcal C}}
\newcommand\Lr{{{\mathcal L}_{\text{rings}}}}
\newcommand\Lrd{{{\mathcal L}^*_{\text{rings}}}}
\newcommand\B{{\mathcal B}}
\newcommand\La{{\mathcal L}}
\def\U{{ \mathfrak U}}
\def\B{{\mathfrak B}}
\def\I{{\mathcal I}}
\def\U{{\mathcal U}}
\def\Ma{{\mathfrak M}}
\def\Ge{{\mathfrak D}}
\let\le=\leqslant
\let\ge=\geqslant
\let\subset=\subseteq
\let\supset=\supseteq
\author{Fran\c coise {Point}$^{(\dagger)}$}
\address{Department of Mathematics (De Vinci)\\ UMons\\ 20, place du Parc 7000 Mons, Belgium}
\email{point@math.univ-paris-diderot.fr}
\thanks{$(\dagger)$ Research Director at the "Fonds de la Recherche
  Scientifique (F.R.S.-F.N.R.S.)"}
\author{Nathalie {Regnault}}
\address{Department of Mathematics (De Vinci)\\ UMons\\ 20, place du Parc 7000 Mons, Belgium}
\email{nathalie.regnault@umons.ac.be}
\title{Exponential topological fields with a generic derivation}
\date{\today}

\begin{abstract}  We axiomatize a class of existentially closed differential expansions of exponential topological fields where the derivation is an $E$-derivation.
 We apply our results to differential expansions of, on the one hand the field of real numbers endowed with $exp(x)$, the classical exponential function defined by its power series expansion, and on the other hand the field of p-adic numbers endowed with the function $exp(px)$ defined on the subring of $p$-adic integers where  $p$ is a prime number strictly bigger than $2$ (or with $exp(4x)$ when $p=2$).
\end{abstract}
\subjclass{03C60, 03C10, 12L12,12H05}
\keywords{exponential field, differential field, existentially closed}
\maketitle
\section{Introduction}
The problem we address here is the following:  given an elementary class of existentially closed exponential topological fields of characteristic $0$ (where possibly the exponential function $E$ is partially defined) whether the class of existentially closed differential expansions is again an elementary class and if this is the case how it can be axiomatized. The model-complete theories of exponential fields we include in our analysis are the theory of the field of real numbers with the exponential function and the field of p-adic numbers with the exponential function restricted to the subring of p-adic integers. 
The derivations $\delta$ we consider are 
$E$-derivations, namely $\delta(E(x))=\delta(x)E(x)$. We answer the question above as follows.

\par We place ourselves in topological fields endowed with a definable $V$-topology, namely a definable topology which is either induced by an archimedean absolute value or a non-trivial valuation \cite[Section 3]{PZ}. The condition that the topology is a $V$-topology is not directly used, but all our examples are endowed with such topology and at some point we assume that our fields satisfy an implicit function theorem which is known to hold for polynomial functions for t-henselian fields (see \cite{PZ} and section \ref{top}).
\par Given an $\L$-theory $T$ of fields, where $\L$ is a relational expansion of the language of fields together with a unary function symbol for the exponential function,
we denote by $T_\delta$ the $\L\cup\{\delta\}$-theory consisting of $T$ together with an axiom expressing that $\delta$ is an $E$-derivation (but $\delta$ is not assumed to be continuous).
\par Then given a model-complete theory $T$ of topological $\L$-fields endowed with a definable $V$-topology, we axiomatize the class of existentially closed expansions by an $E$-derivation of models of $T$, provided the class of models of $T$ satisfy two properties: an implicit function theorem and the lack of flat functions (Theorem \ref{ec}). Those two properties are known to hold when $T$ is the theory of $(\IR,exp)$, or $T$ is the theory of $(\IQ_p,E_p)$, where $\IQ_p$ is the field of p-adic numbers,
or $T$ is the theory of $(\IC_{p},E_{p})$, 
 where $\IC_p$ is the completion of the algebraic closure of $\IQ_{p}$. Note that in these last cases, the exponential function is only partially defined (on the valuation ring) and $E_{2}(x):=exp(4x)$, $E_{p}(x):=exp(px)$, $p\neq 2$.
\par The model-completeness of the theory $T:=Th(\IR,exp)$ was shown by A. Wilkie  \cite{W}, of the theory $T:=Th(\IQ_p,E_p)$ or $T=Th(\IC_{p},E_{p})$ by N. Mariaule \cite{M}, \cite{M1} (section \ref{exa}).

The implicit function theorem is stated in the form of a first-order scheme of axioms $(\IFTA)_{e}$ (for definable functions corresponding to exponential polynomials, see Definition \ref{imp}) but even if we denote the lack of flat function property by the acronym $(\LFF)$ (Definition \ref{lff}), it is not clear that it is a first-order property in general.

The axiomatisation that we give of the existentially closed models of $T_\delta$ is explicit; we call the scheme of axioms $(\DL)_E$ and it expresses the property that when certain systems of exponential polynomials have a regular solution, then they have a differential solution close to that solution (Definition \ref{DL}).

We separate the two issues: whether the axiomatisation we propose indeed axiomatizes the existentially closed differential expansions (Theorem \ref{DLE}) and whether any differential expansion of a model of $T$ can be embedded in a model of this scheme of axioms (Theorem \ref{emb_delta}). To sum up our main result can be stated as follows.
 
\par {\bf Theorem} (later Theorem \ref{ec})
Let $T$ be a model-complete theory of topological $\L$-fields endowed with a definable $V$-topology.
Assume that $\K\models T$ and that the differential expansion $\K_{\delta}$ is a model of $T_\delta\cup (\DL)_{E}$. Then $\K_{\delta}$ is existentially closed in the class of  models of $T_{\delta}$. In particular if the theory $T_\delta\cup (\DL)_E$ is consistent, then it is model-complete.
In case the models of $T$ satisfy an implicit function theorem, namely the scheme $(\IFTA)_{e}$ and have the lack of flat functions property $(\LFF)$, then the theory $T_\delta\cup (\DL)_E$ is consistent.

\

Our main technical tool is the fact that (partial) exponential fields can be equipped with a closure operation $\Ecl$, defined using Khovanskii systems, which has the exchange property, and which coincides with the closure operation $\Cl$ defined using $E$-derivations \cite{K} (section \ref{Khov}). Otherwise we proceed along the same lines as in the model-completeness proof of $(\IR,exp)$.

\

Independently, this question has also been considered by A. Fornasiero and E. Kaplan in the following setting. Given an o-minimal expansion $\K$ of an ordered field which is model-complete and expanded with  a {\it compatible} derivation \cite{FK}, they show that indeed the class of existentially closed differential expansions is elementary and they provide an axiomatization. A derivation $\delta$ is {\it compatible} with $\K$ if for any $0$-definable $C^1$-function $f\colon U\to K$, where $U$ is an open subset of some cartesian product $K^n$, we have $\delta f(\bar u)=\sum_{i=1}^n \frac{\partial f}{\partial x_i}(\bar u) \delta(u_i)$, for any $\bar u\in U$. In particular in case $\K$ expands an exponential field, such derivation $\delta$ is an $E$-derivation. Their results apply to o-minimal fields $\K$ extending the field of real numbers $\IR$ and admitting an expansion to all restricted analytic functions. In order to show that there is a compatible derivation, they have at their disposal the quantifier elimination result of J. Denef and L. van den Dries on the expansion $\IR_{an}$ of $\IR$ with all these functions (with restricted division) and its extension by L. van den Dries, A. Macintyre and D. Marker for $\IR_{an, exp}$, where $exp$ is the exponential function given by the classical power series \cite{DMM2}.

So when $\delta$ is a compatible derivation, in case of $(\IR,exp)$, by uniqueness of the model-completion, one gets, following either approaches, the same class of existentially closed exponential differential fields. However, it is unclear in an ordered exponential field model of the theory of $(\IR, exp)$ whether any $E$-derivation is compatible. (We cannot apply the argument used by A. Fornasiero and E. Kaplan  since we don't have quantifier-elimination in the language of ordered fields together with the exponential function.) 

\

The plan of the paper is as follows.  

In section 2, we review the notion of partial exponential fields and of the corresponding closure operator, denoted by $\Ecl$-closure (Definition \ref{Kho}). It was introduced by A. Macintyre using the work of A. Khovanskii \cite{Mac}, then it plays a crucial role in the proof of A. Wilkie of the model-completeness of $(\IR,exp)$. Later in a purely algebraic context, J. Kirby linked the $\Ecl$-closure with the $\Cl$-closure, defined through $E$-derivations (Definition \ref{cl}). He showed that the two closure operators coincide using a result of J. Ax on the Schanuel property in differential fields of characteristic $0$. Using that last result, we show, under some conditions, how to extend an $E$-derivation (Lemma \ref{der}, Corollary \ref{der-ext-gener}) with some required properties (following the classical case, but replacing algebraic independence by $\Ecl$-independence).
Then in section 3, we recall the notion of $E$-varieties, generic points and torsors. We also recall the setting of fields endowed with a definable topology \cite{P}.
We define the class of differential expansions of exponential fields models of a model-complete theory $T$ for which we can show that the theory we describe and denote by $T_\delta^*$ is consistent, namely those satisfying an implicit function theorem $(\IFTA)_{e}$ (see Definition \ref{imp}) and a property $(\LFF)$ on differential ideals called the lack of flat functions in the ordered case (see Definition \ref{lff}). 
Both properties were shown to hold in o-minimal expansions of real-closed fields (or more generally in definably complete ordered fields), as well as
 in the classes of valued fields mentioned above, as shown by N. Mariaule (see section \ref{sec:imp}). 

In section 4, we finally introduce the theory $T_\delta^*$ consisting of the theory $T$ together with an axiom stating that $\delta$ is an $E$-derivation and a scheme of axioms $(DL)_E$. We show first that if the class of models of $T$ satisfy  $(\IFTA)_{e}$ and $(\LFF)$, then we can embed any model of $T_\delta$ in a model of $T_\delta^*$. Then we show that a model of $T_\delta^*$ is existentially closed in the class of models of $T_\delta$. The scheme of axioms $(\DL)_E$ can be compared to the axiomatization of  M. Singer of the closed ordered differential fields, denoted by $\CODF$. We also give a geometric interpretation of the scheme $(\DL)_E$, which is a priori not first-order.
 
In section 5, we give examples of topological fields to which we may apply our results.

Finally, in the last section, we show how to endow a topological exponential field of cardinality $\aleph_1$ which is first-countable and separable with an $E$-derivation which satisfies this scheme of axioms. When the topology is induced by an ordering we point out that such ordered field can also be made a model of $\CODF$. This kind of construction (for $\CODF$)  may be found in the work of M. Singer, and in the theses of C. Michaux and Q. Brouette.
\medskip
 \par {\bf Acknowledgments:} Part of these results appeared in the PhD thesis of Nathalie Regnault \cite{Re}.
 
\section{$E$-derivations}
\subsection{Preliminaries}\label{prelim}
We only consider commutative rings $R$ of characteristic $0$ with $1\neq 0$. Let $\IN^*:=\IN\setminus\{0\}$, $R^*:=R\setminus\{0\}$. Denote by $I(R)$ the subgroup of the invertible elements of $(R^*,\cdot,1)$. Given an ordered set $(I, <)$, denote $I_{\geq j}:=\{i\in I: i\geq j\}$ (respectively $I_{>j}:=\{i\in I: i>j\}$). 
\par Let $\L_{rings}:=\{+,\cdot,-,0,1\}$ be the language of rings; we will work in different expansions $\L$ of $\L_{rings}$ such as $\L_{E}:=\L_{rings}\cup\{E\}$ and $\L_{E,\delta}:=\L_{rings}\cup\{E,\delta\}$ where $E,\;\delta$ are unary functions. The $\L$-formulas will be possibly with parameters and when we want to specify them we will use $\L(B)$ with $B$ a set of constants. Similarly $\L$-definable sets will possibly be definable with parameters. Our notation for tuples will be flexible: $x$ (respectively $a$) will denote a tuple of variables (respectively a tuple of elements) but sometimes in order to stress that we deal with tuples we will use $\bar x$, respectively $\bar a$, or bold letters e.g. $\x$, $\ba$. In this section we will not make the distinction between an $\L$-structure $\M$ and its domain $M$ whereas from subsection \ref{top} on, we will distinguish them.
\dfn {\rm \cite{D}} An $E$-ring $R$ is a ring equipped with a morphism $E$ from the additive group $(R,+,0)$ to the multiplicative group $I(R)$ satisfying $E(0)=1$ and $\forall x\forall y\;(E(x+y)=E(x)\cdot E(y))$. (So an $E$-ring can be endowed with an $\L_{E}$-structure.) 
An $E$-field is a field which is an $E$-ring.
\edfn
\par We will also consider partial $E$-fields, and so the corresponding language contains a unary predicate for the domain of the exponential function.  We will first define partial $E$-domains.
\dfn Let $F$ be an integral domain, namely a commutative ring with no non-zero zero-divisors.
A partial $E$-domain is a two-sorted structure 
\[((F,+_F,\cdot_{F},0_{F},1_{F}), (A,+_A,0_{A}), E),\] where $(A,+_{A},0_{A})$ is an abelian group and
 \mbox{$E:(A,+_A,0_{A})\rightarrow (I(F), \cdot_F, 1_F)$ is a group morphism}. We identify $(A,+_{A},0_{A})$ with an additive subgroup of $(F,+_{F},0_{F})$ and to stress it, we will denote it by $A(F)$. When the domain of $E$ is clear from the context, we will also simply use the notation $(F,E)$, even though $E$ is only partially defined.
 \par A partial $E$-field $F$ is a partial $E$-domain which is a field.
 A partial $E$-subfield $F_{0}$ is a partial $E$-field which is a two-sorted substructure. We denote by $F_{0}(\bar a)_E$, where $\bar a\subset F$, the smallest partial $E$-subfield of $F$ containing $F_{0}$ and $\bar a$ and by $F_{0}\langle \bar a\rangle_E$ the smallest partial $E$-subring generated by $F_0$ and $\bar a$.
 When $F_0=\IQ$, we denote $\IQ\langle \bar a\rangle_E$ simply by $\langle \bar a\rangle_E$. To make the distinction with the $\L_{rings}$-substructure, we denote by $\IQ[\bar a]$ the subring generated by $\bar a$.
\edfn

Note that in  {\rm \cite[Definition 2.2]{K}}, one uses a stronger notion of partial $E$-fields, namely one requires that $A(F)$ is a $\IQ$-vector space, namely one endows $A(F)$ with scalar multiplications $(\cdot q)_{q\in \IQ}$. Since we don't use that stronger notion, instead when given two partial $E$-fields $F_0\subset F$, we require that $A(F_0)$ is  a pure subgroup of $A(F)$ (see definition below).
\nota \label{pure}  Let $F_0,\;F$ be two partial $E$-fields with $F_0$ a substructure of $F$. Then the subgroup $A(F_0)$ is pure in $A(F)$ iff for any $a\in A(F)$ and $n\in \IN^*$, if $na\in A(F_{0})$, then $a\in A(F_{0})$. 
We use the notation $A(F_0)\subset_1 A(F)$.
\enota

\ex \label{example} $\;$\\
\begin{enumerate}
\item Let $F$ be a partial $E$-field and consider the field of Laurent series $F((t))$ (or more generally a Hahn field (see below)).  Then, regardless of whether we put a topology on $F((t))$, we can always define $exp(x):=\sum_{i\geq 0} \frac{x^i}{i!}$ for $x\in tF[[t]]$. Indeed, by Neumann's Lemma, the element $exp(x)\in F[[t]]$ \cite[chapter 8, section 5, Lemma]{F}.
Then, we extend $E$ on $A(F)\oplus tF[[t]]$ as follows.
 Write $r\in A(F)\oplus tF[[t]]$ as $r_0+r_1$ where $r_0\in A(F)$ and $r_1\in t.F[[t]]$. Define $E$ on $A(F)\oplus tF[[t]]$ as follows: $E(r_0+r_1):=E(r_0) exp(r_1)$. So  $F((t))$ can be endowed with a structure of a partial $E$-field with $A(F((t))):=A(F)\oplus tF[[t]]$.
\item More generally, under the same assumption on $F$, let $(G,+,-,0,<)$ be an abelian totally ordered group, then the Hahn field $F((G))$ can be endowed with a structure of a partial $E$-field, defining $E$ on the elements $r\in A(F)\oplus F((G_{> 0}))$ similarly.
Namely decompose $r$ as $r_0+r_1$ with $r_0\in A(F)$ and $r_1\in F((G_{>0}))$. Then $exp(r_1)\in F((G_{\geq 0}))$ again by Neumann's Lemma and define $E(r):=E(r_{0}) exp(r_{1})$. So $A(F((G)))=A(F)\oplus F((G_{> 0}))$.
\item Let $\bar \IR:=(\IR,+,-,\cdot,0,1, E)$ where $E(x)=exp(x)$ defined above.
\item Let $\bar \IC:=(\IC,+,-,\cdot, 0,1, E)$ where $E(x)=exp(x)$.
\item Let $p$ be a prime number; when $p=2$ set $E_p(x):=exp(p^2x)$ and when $p>2$, set $E_p(x)=exp(px)$. Let $\IC_{p}$ be  the completion of the algebraic closure of the field of $p$-adic numbers $\IQ_{p}$ (in $\IC$). As examples of partial $E$-fields, we have the field of $p$-adic numbers $\bar \IQ_p:=(\IQ_{p},+,-,\cdot,0,1, E_{p})$ or  $\bar \IC_p:=(\IC_{p},+,-,\cdot,0,1,E_{p})$. In these two cases, $E_{p}$ is defined on the valuation ring $\IZ_p$ of $\IQ_{p}$ (respectively on the valuation ring $\O_p$ of $\IC_p$). 
\end{enumerate}
\eex 
\par We will investigate these examples further in section \ref{exa}.
Note that, when the field $F$ is endowed with a field topology and when  $\lim_{n\to \infty} \sum_{i\geq 0}^n \frac{x^i}{i!}$ exists, we can consider the (partial) function $x\mapsto  exp(x):=\lim_{n\to \infty} \sum_{i\geq 0}^n \frac{x^i}{i!}$. One can check that the domain of $exp(x)$ is a subgroup and a $\IQ$-vector space whenever $F$ is closed under roots.
\medskip
\dfn \label{E-der} Let $R$ be a (partial)  $E$-ring. An $E$-derivation $\delta$ is a unary function on $R$ satisfying:
 \begin{enumerate}
\item $\delta(a+b)=\delta(a)+\delta(b)$,
\item  the Leibnitz rule: $\delta(a b)=\delta(a) b+a \delta(b)$,
\item $\delta(E(a))=\delta(a) E(a)$.
\end{enumerate}
We will denote the differential expansion of $R$ by $R_{\delta}$.
 \edfn
 For example, let $F_{\delta}$ be a differential $E$-field ($\delta$ can be the trivial derivation). We have already seen how to extend $E$ on $F[[t]]$. Then we extend $\delta$ on the field of Laurent series $F((t))$ by setting $\delta(t)=1$ and by requiring it to be strongly additive. Then $\delta$ is again an $E$-derivation on $F((t))$. Indeed, for $x\in tF[[t]]$, we have $\delta(exp(x))=\sum_{i\geq 0} \delta(\frac{x^i}{i!})=\delta(x) exp(x)$ and for $x\in F[[t]]$ with $x= r_0+r_1$ where $r_0\in A(F)$ and $r_1\in t F[[t]]$, we have $\delta(E(r_0+r_1))=E(r_0) exp(r_1)\delta(r_1)+\delta(r_0) E(r_0) exp(r_1)=\delta(x) E(x)$. This makes $(F((t)), F[[t]], exp,\delta)$ a differential (partial) $E$-field.
 \  
 
 \nota Let $\delta$ be an $E$-derivation on $R$. For~$m\geqslant 0$ and~$a\in R$, we define 
\[
\delta^m(a):=\underbrace{\delta\circ\ldots\circ\delta}_{m \text{ times}}(a), \text{ with $\delta^0(a):=a$,}
\]
and~$\bar{\delta}^m(a)$ as the finite sequence~$(\delta^0(a),\delta(a),\ldots,\delta^m(a))\in R^{m+1}$. 
\par Similarly, given an element~$\ba=(a_1,\ldots,a_n)\in R^n$, we write
\[\bar{\delta}^m(\ba):=(a_1,\ldots,a_n,\ldots,{\delta}^m(a_1),\ldots,{\delta}^m(a_n))\in R^{(m+1)n}.
\] 
Denote by $\IQ\langle \ba \rangle_{E,\delta}$ the $E$-differential subring of $R$ generated by $\ba$ and $\IQ$.
\enota
 
 In section \ref{Khov}, we will consider in general the problem of extending $E$-derivations but first it is convenient to recall the notion of $E$-polynomials and differential $E$-polynomials.
 \subsection{Free exponential rings}\label{free}
\par The construction of free $E$-rings $\IZ[\X]^E$ on finitely many variables $\X:=(X_{1},\ldots,X_{n})$ (and more generally free $E$-rings $R[\X]^E$ over $(R,E)$) can be found in many places in the literature. We believe it is initially due to B. Dahn. The elements of these rings are called {\it $E$-polynomials} in the indeterminates $\X$.
Here we will briefly recall their construction, following \cite{D} and \cite{Mac}. When $n=1$, we will use the variable $X$ and since we will also use differential $E$-polynomials, we will also allow $\X$ to denote a tuple of countably many variables. 
\par Let $R$ be an $E$-ring. Then the ring $R[\X]^E$ is constructed by stages as follows:
let $R_{-1}:=R$, $R_{0}:=R[\X]$ and $A_{0}$ the ideal generated by $\X$ in $R[\X]$. 
Then $R_{0}=R\oplus A_{0}$. 
Let $E_{-1}=E$ on $R$ composed by the embedding of $R_{-1}$ into $R_{0}$.
\par For $k\geq 0$, set $R_{k}=R_{k-1}\oplus A_{k}$ and 
let $t^{A_{k}}$ be a multiplicative copy of the additive group $A_{k}$.  \par For instance for $k=1$, we get $R_{1}=R_{0}[t^{A_{0}}]$ and $A_{1}$ is a direct summand of $R_{0}$ in $R_{1}$. 
\par Then, put $R_{k+1}:=R_{k}[t^{A_{k}}]$ and let $A_{k+1}$ be the free $R_{k}-$submodule generated by $t^{a}$ with $a\in A_{k}-\{0\}$. 
We have $R_{k+1}=R_{k}\oplus A_{k+1}$. 
\par By induction on $k\geq 0$, one shows the following isomorphism: $R_{k+1}\cong R_{0}[t^{A_{0}\oplus\ldots\oplus A_{k}}],$ using the fact that $R_{0}[t^{A_{0}\oplus\cdots\oplus A_{k}}]\cong R_{0}[t^{A_{0}\oplus\cdots\oplus A_{k-1}}][t^{A_{k}}]$ \cite[Lemma 2]{Mac}.
\par We define the map $E_{k}: R_{k}\rightarrow R_{k+1}$, $k\geq 0$, as follows: $E_{k}(r'+a)=E_{k-1}(r') t^a$, where $r'\in R_{k-1}$ and $a\in A_{k}$.
\par Finally let $R[\X]^E:=\bigcup_{k\geq 0} R_{k}$ and extend $E$ on $R[\X]^E$ by setting $E(f):=E_k(f)$ for $f\in R_k$. It is easy to check that it is well-defined. Let $f\in R_{k+1}$, then $f=f_k+g$ where $f_k\in R_k$ and $g\in A_{k+1}$. So $E(f)=E(f_k) t^g$. By definition $E(f_k)=E_k(f_k)$ and so if $f_k=f_{k-1}+g_k$ with $f_{k-1}\in R_{k-1}$ and $g_k\in A_k$, we have $E(f_{k-1})=E_{k-1}(f_{k-1})t^{g_k}$. Unravelling $f$ in this way, we get that $E(f)=E(f_0)t^{g+g_k+\ldots+g_0}$ with $f=f_0+g_0+\cdots+g_k+g$, $f_0\in R,\;g_0\in A_0,\ldots, g_k\in A_k, g\in A_{k+1}$.

\par Finally note that the above construction can be extended when $R$ is a partial $E$-domain, the only change is that we only define $E(f)$ for $f$ as written above when $f_0\in A(R)$.
\medskip
\par Using the construction of $R[\bold{X}]^E$ as an increasing union of group rings, one can define on the elements of $R[\X]^E$ an analogue of the degree function for ordinary polynomials which measures the complexity of the elements; it takes its values in the class $On$ of ordinals and was described for instance in \cite[1.9]{D} for exponential polynomials in one variable. Here we deal with exponential polynomials in more than one variable and so we follow \cite[section 1.8]{Mac}.
\medskip
\par Let us denote by $totdeg_{\bold{X}}(p)$ the total degree of $p$, namely 
the maximum of $\{\sum_{j=1}^m i_{j}\colon$ for each monomial $X_{1}^{i_{1}}\cdots X_{m}^{i_{m}}$ occurring (nontrivially) in $p$ with $i_1,\ldots, i_m\in \N,\; m\in \IN_{\geq 1}\}$.

\par Then one defines a height function $h$ (with values in $\IN$) which detects at which stage of the construction the (non-zero) element is introduced.
\par Let $p(\X)\in R[\X]^E$, then $h(p(\X))=k$, if $p\in R_{k}\setminus R_{k-1}$, $k>0$ and $h(p(\X))=0$ if $p\in R[\X]$.

\par Using the freeness of the construction, one defines a function $rk$
\[
rk: R[\X]^E\to \IN:
\]
\par If $p=0$, set $rk(p):=0$, 
 \par if $p\in R[\X]\setminus\{0\}$, set $rk(p):=totdeg_{\X}(p)+1$ and 
 \par if $p\in R_{k}$, $k>0$, let $p=\sum_{i=1}^d r_{i} E(a_{i})$, where $r_{i}\in R_{k-1}$, $a_{i}\in A_{k-1}\setminus\{0\}$. Set
$rk(p):=d$.
\par Finally, one defines the complexity function $\ord$
\[
\ord: R[\X]^E\to On
\]  as follows. Write $p\in R_k$ as $p=p_{0}+p_{1}+\cdots+p_{k}$ with $p_0\in R_0$, $p_{i}\in A_{i}$, $1\leq i\leq k$. Define $\ord(p):=\sum_{i=0}^k \omega^i.rk(p_{i})$.
\par Note that if $p_{0}=0$, then there is $q\in R[\X]^E$ such that $\ord(E(q).p)<\ord(p)$ (the proof is exactly the same as the one in \cite[Lemma 1.10]{D}).
\medskip
\par On $R[\X]^E$, we define $n$ $E$-derivations $\partial_{X_{i}}$ as follows:  ${\partial_{X_i}}\restriction R=0$ and $\partial_{X_{i}} X_j= \delta_{ij}$, where $\delta_{ij}$ is the Kronecker symbol, $1\leq i, j\leq n$ \cite[Lemma 3.2]{D}.
\nota\label{dfninitial} Let $\X:=(X_{1},\ldots,X_{n})$ and let $\delta$ be an $E$-derivation on $R$. We consider $R\{\X\}^E$ the  ring of differential $E$-polynomials over $R$ in $n$ differential indeterminates $X_{1},\cdots, X_{n}$, obtained by extending $\delta$ on $R[\X]^E$ first by setting $\delta^{j+1}(X_{i})=\delta(\delta^j(X_{i}))$, $1\leq i\leq n$, $j\in \omega$, with by convention $\delta^0(X_{i}):=X_{i}$. Then by induction for $k>0$ and $p\in R_{k}\setminus R_{k-1}$, writing $p=\sum_{i=1}^d r_{i} E(a_{i})$ with $r_{i}\in R_{k-1}$ and $a_{i}\in A_{k-1}\setminus \{0\}$, we define $\delta(p)=\sum_{i=1}^d (\delta(r_{i})+r_{i} \delta(a_{i})) E(a_{i}).$ 
So $\delta$ is an $E$-derivation by construction and $R\{\X\}^E$ is the $E$-polynomial ring in indeterminates $\delta^j(X_{i}))$, $1\leq i\leq n$, $j\in \omega$. 
\par Let $p(\X)\in R\{\X\}^E$. Let $m\in \IN$ be the (differential) order of~$p$ (denoted by $\delta$-$\ord(p)$) as classically defined in differential algebra \cite[page 75]{Kolchin} (if $m=0$, then $p$ is an ordinary $E$-polynomial). In particular we have that $p$ can be written as $p^*(\bar{\delta}^m(\X))$ with $\bar \delta^m(\X)=(X_{1},\ldots,X_{n},\delta(X_{1}),\ldots,\delta(X_n),\ldots,\delta^m(X_1),\ldots, \delta^m(X_{n}))$ and $p^*$ an ordinary $E$-polynomial.
\enota

\lem \label{der-multi} Let $\delta$ be an $E$-derivation on $R$. Let $p\in R[\X]^E$. Then there exists  $p^{\delta}\in R[\X]^E$ such that in the ring $R\{\X\}^E$, 
$\delta(p)=\sum_{j=1}^n \delta(X_{j}) \partial_{X_j} p +p^{\delta}$. Moreover there is a tuple $\bar e$ of elements of $R$ such that $p\in \langle \bar e\rangle_{E}[\X]^E$ and $p^{\delta}\in \IQ(\langle \bar e\rangle_{E},\delta(\bar e))[\X]^E.$
Furthermore whenever $\delta$ is trivial on $R$, $p^{\delta}=0$.
\elem
\pr Decompose $p$ as: $p=p_{0}+\sum_{i=1}^k p_{i}$, with $p_{0}\in R[\X]$ and $p_{i}\in A_{i}$, $i>0$. We proceed by induction on $\ord(p)$, namely we assume that for all $q\in R[\X]^E$ with $\ord(q)<\ord(p)$, we have $\delta(q(\X))=\sum_{j=1}^n \delta(X_{j}) \partial_{X_j} q+q^{\delta}$ with $q^{\delta}$ satisfying the conditions of the statement of the lemma.
 \par If $ord(p)\in \omega$, namely $p\in R[\X]$, the statement of the lemma is well-known. Write $p(\X)=\sum a_{i_{1},\cdots, i_{n}} X_{1}^{i_{1}}\cdots X_{n}^{i_{n}}$, define $p^{\delta}:=\sum \delta(a_{i_{1},\cdots, i_{n}}) X_{1}^{i_{1}}\cdots X_{n}^{i_{n}}$. Then  $\delta(p(\X))=\sum_{j=1}^n \delta(X_{j}) \partial_{X_j} p+p^{\delta}$. Note that $p^{\delta}\in \delta(R)[\X]$ and $\ord(p^{\delta})\leq \ord(p)$. If $p$ is monic and $n=1$, then $\ord(p^{\delta})<\ord(p)$.
\par Now assume that $ord(p)\geq \omega$ and that the induction hypothesis holds.
\par Let $k>0$ and $p\in R_{k}\setminus R_{k-1}$. By additivity of the derivation, the way $\ord$ has been defined and the induction hypothesis, it suffices to prove it for $p\in A_{k}$.
Let $p=\sum_{i=1}^d r_{i} E(a_{i})$ with $r_{i}\in R_{k-1}$ and $a_{i}\in A_{k-1}\setminus \{0\}.$ 
By definition,  
$\delta(p)=\sum_{i=1}^d (\delta(r_{i})+r_{i} \delta(a_{i})) E(a_{i})$ and by induction hypothesis, $\delta(r_i)=\sum_{j=1}^n \delta(X_{j}) \partial_{X_j} r_{i} +{r_{i}}^{\delta}$ and $\delta(a_i)=\sum_{j=1}^n \delta(X_{j}) \partial_{X_j} a_{i} +{a_{i}}^{\delta}.$
So we get that $\delta(p)=\sum_{j=1}^n \delta(X_{j})\partial_{X_j}(p)+\sum_{i=1}^d E(a_{i})({r_{i}}^{\delta}+r_{i} {a_{i}}^{\delta})$. Define $p^{\delta}:=\sum_{i=1}^d E(a_{i})({r_{i}}^{\delta}+r_{i} {a_{i}}^{\delta})$ $(\dagger)$. 

Let $\be_i$, $\bc_i$ be tuples of elements of $R$ such that $r_i\in \langle \be_i\rangle_{E}[\X]^E$, $a_i\in \langle \bc_i \rangle_{E}[\X]^E$. Then by induction hypothesis, $r_i^{\delta}\in \IQ(\langle \be_i\rangle_{E}, \delta(\be_i)) [\X]^E$, $a_i^{\delta}\in \IQ(\langle \bc_i\rangle_{E},\delta(\bc_i))[\X]^E$. Let $\bar e:=(\be_1,\ldots,\be_d)$ and $\bar c:=(\bc_1,\ldots,\bc_d)$. We have that $p\in \langle \bar e,\bar c\rangle_{E}[\X]^E$ and by $(\dagger)$, \\
$p^{\delta}\in \IQ(\langle \bar e,\bar c\rangle_{E},\delta(\bar e),\delta(\bar c))[\X]^E$ and if $\delta$ is trivial on $R$, then $p^{\delta}=0$.
\qed
 \rem Note that the map sending $p\in R[\X]^E$ to $p^{\delta}$ is an $E$-derivation since the set of $E$-derivations on $R\{\X\}^E$ is an $R\{\X\}^E$-module.
 \erem 
\subsection{Khovanskii systems}\label{Khov}
 \par Let $F_{\delta}$ be an expansion of a partial $E$-field by an $E$-derivation $\delta$ (see Definition \ref{E-der}).  
 Note that in \cite{D}, the condition of being an $E$-derivation was relaxed to: $\delta(E(x))=r \delta(x) E(x)$, for some $r\in R^*$. However if $\delta$ is an E-derivation, then $r \delta$ is also an $E$-derivation, with $r\in R$. More generally, the set of $E$-derivations on $R$ forms a $R$-module.
Using $E$-derivations, J. Kirby defined a closure operator $\Cl$ in $E$-rings and he showed that $\Cl$ induces a pregeometry on subsets of $R$ \cite[Lemma 4.4, Proposition 4.5]{K} (in particular it has the exchange property).

\dfn \cite[Definition 4.3]{K} \label{cl} Let $R$ be a partial $E$-ring and let $A$ be a subset of $R$. Then,
\[
\Cl^R(A):=\{u\in R \colon \delta(u)=0 {\rm \;\;for\;any\;}E-{\rm derivation\;}\delta\;{\rm vanishing\;on\;}A\}.
\]
\edfn
\par If $A\subset R$, then $\Cl^R(A)$ is an $E$-subring and if $R$ is field, it is an $E$-subfield.

\

Note that in the algebraic case, when an element $a$ is algebraic over a subfield endowed with a trivial derivation $\delta$, then $\delta(a)=0$ as well.
Later, we will see an analog of this property in the case of $E$-derivation working with another closure operator, namely $\Ecl$ (see Lemma \ref{der}).

 \nota\label{inj} Let $R$ be an $E$-ring. In section \ref{free}, we recalled the construction of the ring of $E$-polynomials in $\bold{X}:=(X_{1},\ldots,X_{n})$ over $R$. These $E$-polynomials induce functions from $R^n$ to $R$ and we will denote the corresponding ring of functions by $R[\bold{x}]^E$, where $\bold{x}:=(x_{1},\ldots,x_{n})$ \cite{D}. 
\par Note that when $R$ is a partial $E$-domain, we get the same ring of $E$-polynomials but with an $E$-polynomial we can only associate a partially defined function on $R$ (since $E$ is only defined on $A(R)$).

A necessary condition on $R$ under which the map sending an $E$-polynomial $p(\X)$ to the corresponding function $p(\x)$ is injective, is the following:  there exist $n$ $E$-derivations $\partial_{i}$ on $R[\x]^E$, which are trivial on $R$ and satisfy $\partial_{i}(x_{j})=\delta_{ij}$ \cite[Proposition 4.1]{D}. The proof uses the complexity function $\ord$ on exponential polynomials. It holds for instance when $R=\IC$ or $\IR$ \cite[Corollary 4.2]{D}.
 Let $f\in R[\x]^E$, we denote by $\partial_{i} f$, the function corresponding to the differential $E$-polynomial $\partial_{X_{i}} f$.
 \enota
\nota\label{Jac} Given $f_1,\ldots,f_n\in R[\X]^E$, $\bar f:=(f_1,\ldots,f_n)$, we will denote by  $J_{\bar f}(\X)$, the Jacobian matrix: $\left( \begin{array}{ccc}
     \partial_{X_1} f_1 & \cdots & \partial_{X_n} f_1 \\
      \vdots & \ddots & \vdots\\
      \partial_{X_1} f_n & \cdots & \partial_{X_n} f_n  
      \end{array} \right).$
\par As usual, we denote by $det(J_{\bar f}(\X))$ the determinant of the matrix $J_{\bar f}(\X)$; note that  
it is an $E$-polynomial. When we evaluate either $J_{\bar f}(\X)$ or its determinant at an $n$-tuple $\bb\in R^n$, we denote the corresponding values by $J_{\bar f}(\bb)$, respectively $det(J_{\bar f}(\bb))$.
\enota

\dfn \label{Kho} \cite{Mac}, \cite[Definition 3.1]{K} Let $B\subset R$ be partial E-domains. We will adopt the following convention. A {\it Khovanskii system over $B$} is a quantifier-free $\L_E(B)$-formula in free variables $\x:=(x_1,\ldots,x_n)$ of the form 
$$H_{\bar f}(\x):=\bigwedge_{i=1}^n f_i(\x)=0\;\wedge\; det(J_{\bar f}(\x))\neq 0,$$ for some $f_1,,f_n\in B[\X]^E.$
(We will sometimes omit the subscript $\bar f$ in the above formula and possibly make explicit the coefficients $\bar c\in B$ of the $E$-polynomials $\bar f$ in which case, we will use $H_{\bar c}(\x)$.)
\par Let $a\in R$, then $a\in \Ecl^R(B)$ if for some $f_1,\ldots,f_n\in B[\X]^E$ and $a_2,\ldots,a_n\in R$
\[
H_{\bar f}(a, a_2,\ldots,a_n)\;\;{\rm holds}
\]
(assuming that $a_i\in A(R)$, $1\leq i\leq n$, if needed for the $f_i$'s to be defined).
\edfn

\par The operator $\Ecl$ was used by A. Wilkie in his proof of the model-completeness of the theory of $(\bar \IR,exp)$, where $\bar \IR$ denotes the ordered field of real numbers \cite{W}, (see also \cite{JW}). Then J. Kirby extracted $\Ecl$ from this o-minimal setting and showed that it coincides with the closure operator $\Cl$ defined above (see Definition \ref{cl}) \cite[Theorem 1.1, Propositions 4.7, 7.1]{K}. 
Since the operator $\Cl^F$ on subsets of an $E$-field $F$ induces a pregeometry, we get a notion of dimension $\dim^F$ as follows:
\dfn \label{dim}
Let $F$ be a partial $E$-field, let $\x:=(x_{1},\ldots, x_{n})$ and let $C\subset A(F)$ with $C=\Cl^F(C)$, then for $m\leq n$,
\begin{align*}
\dim^F(\x/C)=m\;{\rm if\;}&{\rm there\; exist}\; x_{i_{1}},\ldots, x_{i_{m}}\;{\rm with\;} 1\leq i_{1}<\ldots<i_{m}\leq n\;{\rm such\; that}\\
&\; x_{i_{j}}\notin \Cl(x_{i_{\ell}}, C; 1\leq \ell\neq j\leq m) {\rm\; and}\; x_{i}\in \Cl^F(x_{i_{1}},\ldots, x_{i_{m}}, C), 1\leq i\leq n. 
\end{align*}
\edfn

In order to show that $\Cl\subset \Ecl$, J. Kirby uses a result of J. Ax on the Schanuel property in differential fields of characteristic $0$ \cite[Theorem 5.1]{K}, in order to show the following inequality: 
\[\td(\x, E(\x)/C)-\ell\! \dim_{\IQ}(\x/C)\geq \dim(\x/C),\;\;(\dagger)
\]
where $\td(\x, E(\x)/C)$ denotes the transcendence degree of the field extension $\IQ(\x, E(\x),C)$ of  $\IQ(C)$ and $\ell\! \dim_{\IQ}(\x/C)$ the  dimension of the quotient of two $\IQ$-vector spaces: $\langle\x, C \rangle_{\IQ}$ generated by $\x$ and  $C$, and $\langle C\rangle_{\IQ}$ generated by $C$.
(When $C=\emptyset$, $\ell\! \dim_{\IQ}(\x/C)$ is simply the linear dimension of the $\IQ$-vector-space generated by $\x$.)

\

\par From now on we will not make the distinction between the dimension induced by the closure operator $\Ecl^F$ or by the operator $\Cl^F$.

\

Let $\M_0\subset \M_1$ be two $\L$-structures. Recall that the notation  \mbox{$\M_0\subset_{ec} \M_1$}  means that any existential formula with parameters in $M_0$ satisfied in $\M_1$ is also satisfied in $\M_0$. Let us note some straightforward properties of the $\Ecl^F$ relation (and how it depends on $F$). 

\rem 
Let $F_{0}\subset F_{1}$ be two partial $E$-fields. Suppose that \mbox{$F_{0}\subset_{ec} F_{1}$,} then
\begin{enumerate}
\item $A(F_0)\subset_1 A(F_1)$ (see Notation \ref{pure}),
\item $\Ecl^{F_{1}}(F_{0})=F_{0}$, provided the number of solutions to a Khovanskii system in $F_1$ is finite, and
\item let $\varphi(x_1,\ldots,x_k,\bar y)$ be an existential formula, let $\ba\in F_{0}$, 
then if $\dim^{F_0}(\varphi(F_{0},\ba)/\langle \ba\rangle_E)\geq k$, then $\dim^{F_1}(\varphi(F_{1},\ba)/\langle\ba\rangle_E)\geq k$.
\end{enumerate}
\erem
\pr All these properties are rather straightforward. For convenience of the reader, we will indicate a proof for (2) and (3).
\par (2) Suppose that $u\in \Ecl^{F_{1}}(F_{0})$. So we can find $u_{1},\ldots,u_{n}\in F_{1}$ and a tuple $\bar f$ of $n+1$ $E$-polynomials $f_{1},\ldots, f_{n+1}$ with coefficients in $F_{0}$ such that $H_{\bar f}(u, u_{1},\ldots,u_{n})$ holds. Suppose the number of $n+1$-tuples of solutions of the Khovanskii system $H_{\bar f}$ is finite and equal to $\ell>0$. Then we can express by an existential formula with parameters in $F_{0}$ that it has at least $\ell$ solutions (in $F_{1}$). Since $F_{0}\subset_{ec} F_{1}$, that formula holds in $F_{0}$, so the tuple $(u, u_{1},\ldots,u_{n})$ should appear among the solutions, otherwise we would get one more solution in $F_{1}$, a contradiction.
\par (3) Let $b_{1},\ldots, b_{k}\in F_{0}$, $k>1$, be $\Ecl^{F_0}$-independent over $\langle \ba\rangle_E$ and be such that $\varphi(b_1,\ldots,b_k,\ba)$ holds. Let us show that $b_{1},\ldots, b_{k}$ remain $\Ecl^{F_1}$-independent over $\langle \ba\rangle_E$. We proceed by contradiction assuming for instance that $b_{k}\in \Ecl^{F_{1}}(b_{1},\ldots, b_{k-1}, \langle \ba\rangle_E)$. So there are $\ell\geq 1$ and $E$-polynomials $f_{1},\ldots, f_{\ell}$ with coefficients in $\langle b_{1},\ldots, b_{k-1},\ba\rangle_E$ and $u_{2},\ldots, u_{\ell}\in F_{1}$ such that 
$H_{\bar f}(b_{k},u_{2},\ldots,u_\ell)$ holds. Since $F_{0}\subset_{ec} F_{1}$, we can find $u_{2}',\ldots, u_{\ell}'\in F_{0}$ such that $H_{\bar f}(b_{k},u_{2}',\ldots,u_{\ell}')$ holds, witnessing that $b_{k}\in \Ecl^{F_{0}}(b_{1},\ldots, b_{k-1}, \langle \ba\rangle_E)$, a contradiction.
\qed

\

\par Recall that  in the case of the field of real numbers, A. Khovanskii showed that the number of solutions of a Khovanskii system is not only  finite but it is bounded independently of the coefficients of the system \cite[Proposition 3.1]{W} (he considered the field of real numbers expanded with a Pfaffian chain of functions).

\

\lem \label{der} Let $F_{0}\subset F_{1}\subset F$, where $F_{0}$ is a partial $E$-domain endowed with  an $E$-derivation $\delta$, $F_{1}$ is a partial $E$-field and $F$ is an $E$-field extension of $F_1$, which is endowed with an $E$-derivation extending $\delta$.
Then $\Ecl^{F_1}(F_0)$ can be endowed with a (unique) $E$-derivation extending $\delta$. 
\elem
\pr  Let $u\in \Ecl^{F_1}(F_0)$, so for some $n$, there exist $u_1=u, u_2,\ldots,u_n\in F_1$ such that $H(u_1,\ldots,u_n)$ holds in $F_1$, for some Khovanskii system over $F_0$. Set $\bu:=(u_1,\ldots,u_n)$ and $\X:=(X_1,\ldots,X_n)$.
Let $f_1,\ldots,f_n\in F_0[\X]^E$ be such that 
\begin{equation}
\bigwedge_{i=1}^n f_i(\bu)=0\;\wedge\; det(J_{\bar f}(\bu))\neq 0.
\end{equation} 

Let $\delta^*$ be the given extension of $\delta$ to $F$. 
Applying $\delta^*$ to $f_1(\bu),\ldots,f_n(\bu)$, and using that the $E$-polynomials $f_1^{\delta^*},\ldots,f_n^{\delta^*}$ obtained in Lemma \ref{der-multi} are equal to $f_1^\delta,\ldots,f_n^\delta$, we get 
\begin{equation}
\left( \begin{array}{c}
     f_1^\delta(\bu) \\
      \vdots \\
      f_n^\delta(\bu)   
      \end{array} \right)+J_{\bar f}(\bu)
     \left( \begin{array}{c}
     \delta^*(u_1) \\
      \vdots \\
      \delta^*(u_n)   
     \end{array} \right) =    
     \left( \begin{array}{c}\delta(f_1(\bu)) \\ \vdots \\ \delta^*(f_n(\bu))\end{array}\right)=0
      \end{equation}
So, 
\begin{equation}  \label{3}    
  \left( \begin{array}{c}
     \delta^*(u_1) \\
      \vdots \\
      \delta^*(u_n)  
     \end{array}\right) =-J_{\bar f}(\bu)^{-1}\cdot  \left( \begin{array}{c}
     f_1^\delta(\bu) \\
      \vdots \\
      f_n^\delta(\bu)
      \end{array} \right)
      \end{equation}
      Note that $J_{\bar f}(\X)^{-1}=J^*_{\bar f}(\X)(det(J_{\bar f}(\X))^{-1}$, where $J^*_{\bar f}(\X)$ is the adjugate of $J_{\bar f}(\X)$. So $J_{\bar f}(\X)^{-1}$ is a matrix whose entries are rational $E$-functions with denominator $det(J_{\bar f}(\X)).$ 
      
Since $\Ecl$ has finite character, we may assume that $f_i\in \IQ(\langle \be_i \rangle_E)[\X]^E$ for some tuple $\be_i$ and 
 $f_i^{\delta}\in \IQ(\langle\be_i\rangle_{E}, \delta(\be_{i}))[\X]^E$ (see Lemma \ref{der-multi}). 
 Let $\bar e:=(\be_1,\ldots,\be_n)$;    
  so we can express each $\delta(u_i),\; 1\leq i\leq n$, as an $E$-rational function $t_{i,\bar f}(\bu)$ with coefficients in $\IQ(\langle \bar e, \delta(\bar e)\rangle_{E})$. So $\delta(u_{i})$ still belongs to $\Ecl^{F_{1}}(F_{0})$ since $\Ecl^{F_{1}}(F_{0})$ is a partial $E$-subfield of $F_{1}$.

 \par We can also express the successive derivatives $\delta^{\ell}(u_i)$, $1\leq i\leq n ,\;\ell\in \IN$, $\ell\geq 2$,  as $E$-rational function $t_{i,\bar f}^{\ell}(\bu)$ with coefficients in $\IQ(\langle \bar\delta^\ell(\bar e)\rangle_{E})$.
 Note that the $E$-polynomial appearing in the denominator is a power of $det(J_{\bar f}(\X))$. We set $t_{i,\bar f}^{1}(\bu)=t_{i,\bar f}(\bu)$.
 We can proceed in the same way for every element of $\Ecl^{F_{1}}(F_{0})$ and so it is endowed with an $E$-derivation extending $\delta$ on $F_{0}$. By construction it does not depend on the extension $F$ (also since $\Ecl=\Cl$ there is only one such $E$-derivation extending $\delta$ on $F_0$). \qed

\

\par For later use, we need to make explicit the form of the rational functions $t_{i,\bar f}^{\ell}(\bu)$ as a function of $\bu$ but also of the coefficients of $\bar f$ (see section \ref{ec}). 
\nota\label{rat} By equation (3), we have $\delta(y_i^0):=t_{i,\bar f}(\y^0)$ where $\y^0:=(y_1^0,\ldots,y_n^0)$ and $t_{i,\bar f}(\y^0)$ is obtained by multiplying the matrix $-J_{\bar f}(\y^0)^{-1}$ by the column vector 
$ \left( \begin{array}{c}
     f_1^\delta(\y^0) \\
      \vdots \\
      f_n^\delta(\y^0)   
      \end{array} \right).$
 
 Now by Lemma \ref{der-multi}, there are tuples $\x_i^0\in F_0$ such that $f_i$ belongs to $\IQ\langle\x_i^0\rangle_E[\X]^E$ and
$f_i^{\delta}\in \IQ(\langle \x_i^0\rangle_{E}, \delta(\x_i^0))[\X]^E$.  To $f_i^{\delta}$, we associate an $E$-rational function $f_i^{\delta,*}$ by replacing $\delta(\x_i^0)$ by the tuple $\x_i^1$.
 
Let $\bar x^j:=(\x_1^j,\ldots,\x_n^j)$ with $0\leq j$. Then we re-write $t_{i,\bar f}(\y^0)$ as an $E$-rational function with coefficients in $\IQ$, namely as $t_{i,\bar f}^*(\y^0; \bar x^0, \bar x^1))$, $1\leq i\leq n$. Set $t_{i,\bar f}^{1,*}:=t_{i,\bar f}^*$ and $\ta_{\bar f}^{1,*}:=(t_{1,\bar f}^{1,*},\ldots, t_{n,\bar f}^{1,*})$.
Then we define $t_{i,\bar f}^{2,*}$ by applying $\delta$ and substituting $t_{j,\bar f}(\y^0)$ to $\delta(y_j^0),$ $1\leq j\leq n$, and $\bar x^j$ to $\delta(\bar x^{j-1})$, $2\geq j\geq 1$. 
 So we get an $E$-rational function $t_{i,\bar f}^{2,*}(\y^0; \bar x^0, \bar x^1,\bar x^2)$, $1\leq i\leq n$. We iterate this procedure, namely we apply $\delta$ to $t_{i,\bar f}^{\ell,*}$, we substitute $t_{k,\bar f}^{1,*}(\y^0)$ to $\delta(y_k^0),$ $1\leq k\leq n$, and $\bar x^{j+1}$ to $\delta(\bar x^j)$, $j\geq 0$, to obtain $t_{i,\bar f}^{\ell+1,*}(\y^0; \bar x^0,\ldots,\bar x^{\ell+1}),$ $1\leq i\leq n$. We denote $\ta_{\bar f}^{\ell+1,*}:=(t_{1,\bar f}^{\ell+1,*},\ldots, t_{n,\bar f}^{\ell+1,*})$.
\enota 

\cor \label{der-ext-gener} Let $F_{0}\subset F_1$ be two partial $E$-fields and let $\delta$ be an $E$-derivation on $F_0$.  Assume that we have an $E$-field extension $F$ of $F_1$ which is endowed with an $E$-derivation $D$ extending $\delta$. 
\par Let $c_1,\ldots, c_\ell\in F_1$ are $\Ecl^{F_1}$-independent over $F_0$ and $d_1,\ldots, d_\ell\in F_0$, then there is on $F_1$ an $E$-derivation $\tilde \delta$ extending $\delta$ such that $\tilde \delta(c_i)=d_i$, $1\leq i\leq \ell$.
\par Let $c_1,\ldots, c_\ell\in F_1$ be $\Ecl^{F}$-independent over $F_0$. Then, for any choice of $d_1,\ldots, d_\ell\in F_1$, there is an $E$-derivation $\tilde \delta$ (on $F$) extending $\delta$ on $F_{0}$ and such that $\tilde \delta(c_i)=d_i$, $1\leq i\leq \ell$. 
\ecor
\pr Let $c_1,\ldots, c_\ell\in F_1$ be $\Ecl^{F_1}$-independent over $F_0$ and  $d_1,\ldots, d_\ell\in F_0$. There are $\ell$ E-derivations $\delta_i$ on $F_{1}$ which are zero on $F_0$ and such that $\delta_i(c_j)=\delta_{ij}$, $1\leq i,\; j\leq \ell$ \cite[Proposition 7.1]{K}.
Set $\tilde F$ be the field of fractions of $\langle F_0, \bc\rangle_E\subset F_{1}$, with $\bc:=(c_{1},\ldots,c_{\ell})$ and consider $\tilde \delta:={D}_{\restriction \tilde F}+\sum_{i=1}^{\ell} (d_i-D(c_i)) {\delta_i}_{\restriction \tilde F}$. We have $\tilde\delta(\tilde F)\subseteq F_0$. 
Then we consider a maximal partial $E$-field extension $F_2$ of $\tilde F$ included in $F_1$ and endowed with a derivation $\delta_2$ extending $\tilde \delta$.
By the preceding lemma $F_2=\Ecl^{F_1}(F_2)$. If it is properly included in $F_{1}$, we can choose an element $a\in F_1\setminus \Ecl^{F_1}(F_2)$ and construct an $E$-derivation extending $\delta_2$ on the $E$-subring generated by $F_2$ and $a$, sending $a$ to $1$ (or to any element of $F_2$). Then we extend that derivation on the field of fractions of $\langle F_2, c\rangle_E$, contradicting the maximality of $F_2$.
\par Now let $d_1,\ldots, d_\ell\in F_1$ and assume $c_1,\ldots, c_\ell\in F_1$ are $\Ecl^{F}$-independent over $F_0$. Again there are $\ell$ E-derivations $\delta_i$ on $F$ which are zero on $F_0$ and such that $\delta_i(c_j)=\delta_{ij}$, $1\leq i,\; j\leq \ell$. 
 
We define $\tilde \delta$ as $D+\sum_{i=1}^{\ell} (d_i-D(c_i))\delta_i$; since the set of $E$-derivations on $F$ forms an $F$-module, this is an $E$-derivation which extends $\delta$ by construction and which sends $c_{i}$ to $d_{i}$, $1\leq i\leq \ell$. 
Let $F_{1}'$ be the field of fractions of $\langle F_0, \overline{\tilde\delta}^{n}(\bc); n\in \IN\rangle_E$; it is a partial E-field closed under $\tilde \delta$. We proceed as in the first part of the proof in order to show that a maximal partial $E$-field extension of $F_{1}'$ endowed with a derivation extending $\tilde \delta$, is equal to $F$.
\qed 

\section{E-varieties and topological exponential fields}
\subsection{E-varieties}
Let $K$ be a (partial) $E$-field. Let $f\in K[\X]^E$, $\X:=(X_1,\ldots,X_n)$, and $\ba\in K^{n}$. Denote by $\nabla f:= (\partial_{X_1} f(\X), \ldots,  \partial_{X_n} f(\X))$
 and $\nabla f(\ba):= (\partial_{X_1} f(\ba), \ldots,  \partial_{X_n} f(\ba)).$ 

\dfn \label{var} Let $g_{1},\ldots, g_{m}\in K[\X]^E$ and let 
\[V_{n}(g_{1},\ldots,g_{m}):=\{\ba\in K^{n}\colon \;\bigwedge_{i=1}^{m} g_{i}(\ba)=0\}.
\]
By $E$-variety (defined over $K$), we mean a definable subset of some $K^n$ of the form $V_{n}(\bar g)$ for some $\bar g\in K[\X]^E$. When we consider the elements of an $E$-variety in an extension $\tilde K$ of $K$, we denote the set of these elements by $V(\tilde K)$. Let $V$ be an $E$-variety, then $\ba$ is a regular point of $V$ if for some $\bar g$, $V=V_{n}(\bar g)$ and $\nabla g_{1}(\ba),\ldots,\nabla g_{m}(\ba)$ are linearly independent over $K$ (note that this implies that $m\leq n$).
\edfn
\par In the following, we will make a partition of variables of the $g_{i}$'s and consider the regular zeroes with respect to a subset of the set of variables.
\nota \label{gradient}
Let $0<n_{0}\leq n$  and let $f\in K[\X]^E$. Denote by 
\begin{equation}
\nabla_{n_{0}}  f:=(\partial_{X_{n-n_{0}+1}} f,\ldots, \partial_{X_{n}} f).
\end{equation}
\par Consider the following subset of $V_{n}(\bar g)$, with $m\leq n_0$: 
\begin{equation}
V_{n,n_{0}}^{reg}(\bar g):=\{\bb\in K^n\colon \bigwedge_{i=1}^m g_{i}(\bb)=0\;\&\;\nabla_{n_{0}} g_{1}(\bb),\ldots,\nabla_{n_{0}} g_{m}(\bb) \;{\rm are}\; K\!-{\rm linearly\; independent} \}.
\end{equation} 
Note if $m> n_{0}$, then $V_{n,n_{0}}^{reg}(\bar g)= \emptyset$. In case $n_{0}=n$, we simply denote $V_{n,n}^{reg}(\bar g)$ by $V_{n}^{reg}(\bar g)$ (or $V^{reg}(\bar g)$). 
\enota
\par Furthermore, we need the following variant. 
Let $\bar i:=(i_{1},\ldots,i_{n_{0}})$ be a strictly increasing tuple of natural numbers between $1$ and $n$ (of length $1\leq n_{0}\leq n$). Then for $f\in K[\X]^E$, we denote by 
\begin{equation}
\nabla_{\bar i}  f:=(\partial_{X_{i_{1}}} f,\ldots, \partial_{X_{i_{n_{0}}}} f).
\end{equation}
We consider the following subset of $V_{n}(\bar g)$, with $m\leq n_{0}=\vert \bar i\vert$: 
\begin{equation}
V_{\bar i}^{reg}(\bar g):=\{\bb\in K^n\colon \bigwedge_{i=1}^m g_{i}(\bb)=0\;\&\;\nabla_{\bar i} g_{1}(\bb),\ldots,\nabla_{\bar i} g_{m}(\bb)\;{\rm are}\; K\!-\!{\rm linearly\; independent} \}.
\end{equation} 
\par We will denote the formulas corresponding to these definable sets by $\x\in V_n(\bar g)$, respectively $\x\in V_n^{reg}(\bar g)$.

\subsection{Generic points}
\par Let $K\subseteq L$ be partial $E$-fields. In section \ref{Khov}, we have seen that $\Ecl^L$ is a closure operator which coincides with $cl^L$ to which we associated the dimension function $\dim^L(\cdot/K)$ (see Definition \ref{dim}). As usual one defines the dimension of a definable subset $B\subset L^n$ and the notion of generic points in $B$. 
\dfn \label{dim2} Let $B$ be a definable subset of $L^n$ defined over $K$. The dimension of $B$ over $K$ is defined as
$\dim^L(B/K):=\sup\{\dim^L(\bb/K)\; : \; \bb\in B\}$. Let $\bb\in B$, then $\bb$ is a generic point of $B$ over $K$ if $\dim^L(\bb/K)=\dim^L(B/K)$. 
\edfn

We will need the following notion of subtuples. 
\nota\label{subtuple}
Let $\ba:=(a_{1},\ldots,a_{n})$ be an $n$-tuple in $K$ and let $\X:=(X_1,\ldots,X_n)$. Let $0<m<n$ and let $\{i_{1},\ldots,i_{m}\}\dot\cup\{j_{1},\ldots,j_{n-m}\}$ be a partition of $\{1,\ldots,n\}$, with $1\leq i_{1}<\ldots<i_{m}\leq n$ and $1\leq j_{1}<\ldots<j_{n-m}\leq n$. 

A {\it $m$-subtuple} of $\ba$ is a $m$-tuple denoted by $\ba_{[m]}$ of the form $(a_{i_{1}},\ldots,a_{i_{m}})$ and we denote by $\ba_{[n-m]}:=(a_{j_{1}},\ldots,a_{j_{n-m}})$. 

Given an $E$-polynomial $f(\X)\in K[\X]^E$, we denote either by $f(\ba_{[n-m]},X_{i_{1}},\ldots,X_{i_{m}})$ or by $f_{\ba_{[n-m]}}(X_{i_{1}},\ldots,X_{i_{m}})$ the $E$-polynomial obtained from $f$ when substituting  for $X_{j_{i}}$ the element $a_{j_{i}}$, $1\leq i\leq n-m$. We adopt the same convention for $\L_{E}$-terms.
\enota
\rem\label{polmin}
 Let $\bar f=(f_1,\ldots,f_m)\subseteq K[\X]^E$, $\ba:=(a_1,\ldots,a_n)\in V_n^{reg}(\bar f)\subseteq L^n$, $1\leq m\leq n$. Then:
 \begin{enumerate}
 \item There is a $m$-subtuple $\ba_{[m]}$ of $\ba$ and a Khovanskii system over $K\langle \ba_{[n-m]}\rangle^E$ such that $H_{\bar f_{\ba_{[n-m]}}}(\ba_{[m]})$ holds.
 \item In particular $\dim^L(\ba/K)\leq n-m$ and if $V_n(\bar f)=V_n^{reg}(\bar f)$, then $\dim^L(V_n(\bar f)/K)\leq n-m$.
 \end{enumerate}
\erem
\subsection{$E$-ideals and differentiation}
\par Let $R$ be a partial $E$-ring. Let $\X:=(X_{1},\ldots, X_{n})$
and $\X_{\hat i}$ be the tuple $\X$ where $X_{i}$ is removed, $1\leq i\leq n$. Similarly for $\ba\in R^n$, we denote $\ba_{\hat{i}}:=(a_1,\ldots,a_{i-1},a_{i+1},\ldots,a_{n})$.

 \dfn  Let $I\subseteq R$ be an ideal of $R$. Then $I$ is an $E$-\textit{ideal} if 
 \[
 (r\in I\rightarrow E(r)-1\in I).
 \]
 A prime $E$-ideal is a prime ideal which is an $E$-ideal.
 \edfn
In $R[\X]^E$, an example of a prime $E$-ideal is $Ann^{R[\X]^E}(\ba):=\{f\in R[\X]^E:\;f(\ba)=0\}$. (When the context is clear we will omit the superscript $R[\X]^E$.) 
\smallskip
\par As usual the definition of $E$-ideal is set-up in such a way that if $I\subseteq R$ is an $E$-ideal, then on the quotient $R/I$, we have a well-defined exponential function given by: 
$$E(r+I):=E(r)+I$$
for $r\in A(R)$. So  $(R/I, E)$ is a again a partial $E$-ring.
\medskip 
\par We now recall a result from A. Macintyre on $E$-ideals closed under partial derivation. Note that the proof is purely algebraic, using that one can measure the complexity of exponential polynomials (see section \ref{free}).
\fct \cite [Theorem 15 and Corollary]{Mac}\label{dpnul} Let $R$ be a partial $E$-domain. Let $1\leq i\leq n$.
Let $I\subseteq R[\X]^E$ be an $E$-ideal closed under the $E$-derivation $\partial_{X_{i}}$.
 Then either $I=0$ or $I$ contains a non-zero element of $R[\X_{\hat i}]^E$.  
In particular, if $I\neq 0$ is closed under all $E$-derivations $\partial_{X_{i}}$, $1\leq i\leq n$ and $R$ is a field, then  $I=R[\X]^E$.
\efct

Let $K\subseteq L$ be  partial $E$-fields. A consequence of Fact \ref{dpnul} is that $\Ecl^L$-independent elements over $K$ do not satisfy any hidden exponential-algebraic relations over $K$.
\cor\label{norel} Let $\ba:=(a_1,\ldots,a_n)\in  L^n$ be such that $a_1,\ldots,a_n$ are $\Ecl^L$-independent over $K$. 
Then there is no $g\in K[\X]^E\setminus\{0\}$ such that $g(\ba)=0$.  
\ecor
\begin{proof} By the way of contradiction assume there is $g\in K[\X]^E\setminus\{0\}$ be such that $g(\ba)=0$.
  Then for  $i=1,\ldots,n$, $\partial_{X_i} g(\ba)=0$ otherwise $a_i\in ecl^L(K(\ba_{\hat i}))$. (Indeed, letting $h(X):=g(\ba_{\hat{i}},X)$, we would have $H_{h}(a_{i})$.) Hence the ideal $Ann(\ba)$ is an $E$-ideal, closed under all partial $E$-derivations $\partial_{X_{i}}$, $1\leq i\leq n$. So by Fact \ref{dpnul}, since $Ann(\ba)\neq 0$, it is equal to $K[\X]^E$, a contradiction.
\end{proof}

\

Let $K_{\delta}$ be an expansion of the partial $E$-field $K$ by an $E$-derivation $\delta$ and let $\tilde K$ be an $E$-field extending $K$. Let $A\subset \tilde K^n$.
Let $I_K(A)\subseteq K[\X]^E$ be the set of $E$-polynomials with coefficients in $K$ which vanish on $A$, namely $I_K(A)=\bigcap_{\ba\in A} Ann^{K[\X]^E}(\ba)$. Note that it is an $E$-ideal as an intersection of $E$-ideals. 
\begin{definition} \label{torsor} For $A\subset \tilde K^n$, let $\tau(A)\subset \tilde K^{2n}$ be the \textit{$E$-torsor} of $A$ (over $K$), namely:
\[\tau(A):=\{(\ba,\bb)\in \tilde K^{2n}:\ba\in A \textrm{\,and\,} \sum_{i=1}^n \partial_{X_{i}} f(\ba) b_i+f^{\delta}(\ba)=0 \textrm{\;for\;all\;}f(\X)\in I_K(A)\}.
\]
\end{definition}
In case $\tilde K$ is endowed with a derivation $\tilde \delta$ extending $\delta$, if $\ba\in A$, then $(\ba,\tilde \delta(\ba))\in \tau(A)$.
\medskip
\par Note that if we can find $f_1(\X),\ldots, f_m(\X)\in Ann^{K[\X]^E}(\ba)$, $\ba\in A$, $1\leq m\leq n$ such that $\nabla f_1(\ba),\ldots,\nabla f_m(\ba)$  are $K$-linearly independent, then setting
\[ 
T_{\ba}:=\{\bb\in \tilde K^{n}:\;\sum_{i=1}^n \partial_{X_{i}} f(\ba)  b_i=0 \textrm{\;for\;all\;}f(\X)\in Ann^{K[\X]^E}(\ba)\},
\]
we have that $\ell \dim(T_{\ba})\leq n-m$. 

\begin{lemma}\label{Lang-tor} Let $K\subset_{ec} \tilde K$ be partial $E$-fields, and let $\delta$ be an $E$-derivation on $K$. Let $\bar f=(f_1,\ldots, f_m)\subseteq K[\X]^E$, $m\leq n$. Suppose that there are $(\ba,\bb)\in \tilde K^{2n}$
such that $(\ba,\bb)\in \tau(V_{n}^{reg}(\bar f))$ and $\dim^{K^*}(\ba)=n-m$.
Then there is an $E$-derivation $\tilde \delta$ (on an elementary extension of $K$, as partial $E$-fields) extending $\delta$ such that $\tilde\delta(a_i)=b_i$, for $i=1,\ldots,n$.
\end{lemma}
\begin{proof} Since $K\subset_{ec} \tilde K$ as partial $E$-fields, there is an embedding (of partial $E$-fields) of $\tilde K$ in a non-principal ultrapower $K^*$ of $K$ (which is the identity on $K$). Let $\delta^*$ be the derivation induced by $\delta$ on $K^*$.
\par Since $\ba\in V_{n}^{reg}(\bar f)$, we have that $\nabla f_{1}(\ba), \ldots, \nabla f_{m}(\ba)$ are $\tilde K$-linearly independent.
By permuting the coordinates of $\ba$, assume $\nabla_{m}f_{1}(\ba),\ldots,\nabla_{m}f_{m}(\ba)$ are $\tilde K$-linearly independent.
Set $\ba_{[n-m]}:=(a_{1},\ldots,a_{n-m})$ and $\ba_{[m]}:=(a_{n-m+1},\ldots,a_{n})$. Note that $det(J_{\bar f_{\ba_{[n-m]}}}(\ba_{[m]}))\neq 0$. Since  $\ba_{[m]}\in \Ecl^{\tilde K}(\ba_{[n-m]})$, $a_{1},\ldots, a_{n-m}$ are $\Ecl^{K^*}$-independent. 

By Corollary \ref{der-ext-gener}, there is an $E$-derivation $\tilde \delta$ on $K^*$ extending $\delta$ on $K$ and such that $\tilde \delta(a_{i})=b_{i}$, $1\leq i\leq n-m$. 

By assumption $(\ba,\bb)\in \tau(V_{n}^{reg}(\bar f))$. In particular $\bigwedge_{i=1}^m \sum_{j=1}^n \partial_{X_{j}} f_{i}(\ba) b_j+f_{i}^{\delta}(\ba)=0$. Breaking the sum $\sum_{j=1}^n \partial_{X_{j}} f_{i}(\ba) b_j$ in two parts: $\sum_{j=1}^{n-m} \partial_{X_{j}} f_{i}(\ba) b_j$, $\sum_{j=n-m+1}^{n} \partial_{X_{j}} f_{i}(\ba) b_j$, we get:

\begin{equation}  \label{4}    
  \left( \begin{array}{c}
     b_{n-m+1} \\
      \vdots \\
     b_n  
     \end{array} \right) =-(J_{\bar f_{\ba_{[n-m]}}}(\ba_{[m]}))^{-1}  \left( \begin{array}{c}
     f_1^\delta(\ba)-\sum_{j=1}^{n-m} \partial_{X_{j}} f_{1}(\ba) b_j \\
      \vdots \\
      f_m^\delta(\ba)- \sum_{j=1}^{n-m} \partial_{X_{j}} f_{m}(\ba) b_j 
      \end{array} \right)
      \end{equation}
 Since $\tilde \delta(a_{i})=b_{i}$, $1\leq i\leq n-m$, we get that $\tilde\delta(\ba)=\bb$.   
\end{proof}
\rem Keeping the same notation as the above lemma, we will show later (see Proposition \ref{generegalter}) that, under the condition that $\K$ satisfies an implicit function theorem (see Definition \ref{imp}), in a sufficiently saturated elementary extension of $K$, we can find a tuple $\ba\in V_{n}^{reg}(\bar f)$ with $\dim^{K^*}(\ba)=n-m$.
\erem

\subsection{Topological $E$-fields} \label{top} In section 2, we introduced the notion of a partial $E$-field $F$ as a  two-sorted structure $(F,A(F),E)$ where $F$ is a field and $A(F)$ is an additive group (that we identified with a subgroup of the additive group of $F$) and $E:A(F)\to F^*$, a morphism from $A(F)$ to the multiplicative group of $F$. 

\par In this section we will revert to a one-sorted setting and we work with topological fields, namely the field operations, together with the exponential function are continuous (when defined). 
Let $\L^-$ be a relational extension of $\L_{E}$ and let $\L:=\L^-\cup \{{}^{-1}\}$. Starting with a two-sorted structure $(K,A(K),E)$ which is a partial $E$-field, we will consider $\L$-structures $\K$ with domain $K$ and with the convention that the domain of $E$ is $A(K)$, an additive subgroup of $K$. 
Classically one requires that the functions are defined everywhere and so for instance, one extends ${}^{-1}$ by the rule $0^{-1}=0$. 
Instead here we proceed as follows. We use the fact that one can associate with any $\L$-term $t(x_1,\ldots,x_n)$ a quantifier-free formula $D_t(x_1,\ldots,x_n)$ which exactly holds on the domain of definition (also denoted by $D_t$) of $t$ (see for instance \cite[Section 2]{Wolter}). (One works by induction on the complexity of terms in a similar way as we did in section 2.) Furthermore letting $\partial_i t(x_1,\ldots,x_n)$ be the formal derivative of $t$ with respect to $x_i$ (with the rule $\partial (E(x))=E(x)$), $D_t(x_1,\ldots,x_n)\leftrightarrow D_{\partial_i t}(x_1,\ldots,x_n)$, $1\leq i\leq n$. The advantage of proceeding in this way, instead of extending the functions when undefined, is that one can then require that the terms induce continuous functions (or continuously differentiable (i.e. C$^1$-functions), or C$^{\infty}$, or analytic functions) on their domains of definition.
\par  Let $\V$ denote a basis of neighbourhoods of $0$. Then $(\K,\V)$ is a topological $\L$-field if $\V$ induces an Hausdorff (non-discrete) topology such that the functions of $\L$ are interpreted by C$^1$-functions on their domains of definition and that each relation and its complement is the union of an open set and an $E$-variety.
So w.l.o.g. we may assume that every quantifier-free $\L$-formula is a finite disjunction of quantifier-free formulas of the form:
\[
\bigwedge_{j\in J} t_j(\x)=0\wedge \x\in O,
\]
where $t_j(\x)$ is an $\L$-term, $j\in J$ ($J$ finite) and $O$ is a definable open set by a conjunction of basic formulas.
We will say that $\K$ is endowed with a definable topology if there is an $\L$-formula $\chi(x,\y)$ such that a basis of neighbourhoods of $0$ in $K$ is given by $\chi(K,\bd)$, where $\bd\in K^n$, $n=\vert \y\vert$. 
Note that if $\K$ is endowed with a definable topology, then any field $\K_{0}$ elementary equivalent to $\K$ can be endowed with a definable topology using the same formula $\chi(x,\y)$. 
Moreover if $\K$ is endowed with a definable topology with corresponding formula $\chi(x,\y)$ and $\tilde \K$ an elementary extension of $\K$ endowed with a topology induced by $\chi$, then $\tilde \K$ is a topological extension of $\K$ \cite[Definition 2.3]{GP}.
As usual, the cartesian products of $K$ are endowed with the product topology. Let $\x$ be a $m$-tuple, we will denote by $\bar \chi(\x,\y)$ the formula $\bigwedge_{i=1}^m \chi(x_{i},\y)$.

\nota \label{Ksmall} Let $(\K,\V)\subset (\tilde \K,\W)$ be two topological $\L$-fields with $(\tilde \K,\W)$ be a topological extension of $(\K,\V)$, as defined in \cite[Definition 2.3]{GP}, namely $\K$ is an $\L$-substructure of $\tilde \K$ and for any $V\in \V$ there exists $W\in \W$ such that $V=W\cap K$. Let $\W_{K}:=\{W\in \W: W\cap K\in \V\}$.
\par On elements  $a, b\in \tilde K$ we have the equivalence relation $a \sim_{\W_{K}} b$ which means that $a-b$ belongs to every element of $\W_{K}$. (We will also use the notation $a\sim_{K}b$.)
\par We will say that a non zero element $a\in\tilde K$ is $K$-small if $a \sim_{\W_{K}} 0$ (that we abbreviate by $a\sim_K 0$). 
\enota 

\par Let $K$ be a topological field. Then a subset $Z\subset K$ is bounded away from $0$ if $0$ does not belong to the closure of $Z$. The topology on $K$ is a $V$-topology if whenever $X, Y\subset K$ are bounded away from $0$, then $X Y$ is bounded away from $0$. One calls such fields $V$-topological fields \cite[Section 3]{PZ}. By results of Kowalsky-D\"urbaum, and Fleischer if $K$ is a $V$-topological field then its topology is either induced by an archimedean absolute value or by a non-trivial valuation \cite[Theorem 3.1]{PZ}. One can define a notion of topological henselianity (t-henselianity) for $V$-topological fields \cite[Theorem 7.2]{PZ}. One can show that one can embed any $V$-topological field in a t-henselian field \cite[Lemma 2.2]{HHJ} and a t-henselian field satisfies the implicit function theorem for polynomial maps \cite[Theorem 7.4]{PZ}, \cite[Fact 2.4]{HHJ}. In the following, we will need an analog of that property for exponential polynomials.

\subsection{Implicit function theorem}\label{sec:imp}
From now on, we will assume that $\K$ is a topological $\L$-field where the topology is a $V$-topology and it is definable with corresponding formula $\chi$. 
\nota \cite[Definition 4.4]{W} \label{neigh} Let $\Se$ be a neighbourhood system in $K^n$, namely a non-empty collection of open non-empty definable  neighbourhoods closed under finite intersection. Let $\ba\in K^n$, we will denote by $\Se_\ba$ the neighbourhood system consisting of all definable neighbourhoods of $\ba$.
\par Denote by ${\Ge^n(\Se)}^-:=\{(f,U)\colon U\in \Se, f: U\to K$ a $C^{\infty}$-function, definable in $\K \}.$ One defines on ${\Ge^n(\Se)}^-$ an equivalence relation $\sim$ as follows: $(f_1,U_1)\sim (f_2,U_2)$ if there is $U\subset U_1\cap U_2$ such that $f_1\restriction U=f_2\restriction U$. Let $\Ge^n(\Se):={\Ge^n(\Se)}^-/\sim$. We denote by $[f,U]$ the equivalence class containing $(f,U)$.
\par Denote by ${{\Ge_{an}^n}(\Se)}^-:=\{(f,U)\colon U\in \Se, f: U\to K$ an analytic function, definable  in $\K\; \}$ and by $\Ge_{an}^n(\Se):={\Ge_{an}^n(\Se)}^-/\sim$.
\enota

\medskip
\par We now introduce the following {\it implicit function theorem} hypothesis $(\IFTA)_{e}$ that we put on the class of fields under consideration. The implicit function theorem for $\text{C}^1$-functions, or 
C$^{\infty}$-functions, or analytic functions is classically proven in fields like $\IR$, $\IQ_p$ (or more generally complete (non-discrete) valued fields (of rank $1$)) \cite[section 1.5]{B}. A. Wilkie stated it for any field $\K$ elementary equivalent to an expansion of the field of reals \cite[section 4.3]{W}, T. Servi recasted the results of Wilkie in definably complete expansions of ordered fields \cite{S}.
\dfn\label{imp}  Let $n=\ell+m$, $n>1$, $\ell, m>0$, let $(\ba,\bb)\in K^{\ell+m}$. 
 Let $f_1(\x,\y),\ldots, f_m(\x,\y)$ be definable $C^{1}$-functions in $\K$, $\vert \x\vert=\ell$, $\vert \y\vert=m,$
denote $\bar f(\ba, \y)=(f_{1}(\ba,\y),\ldots, f_{m}(\ba,\y))$ by $\bar f_{\ba}(\y)$. Then $\K$ satisfies $(\IFTA)$ if the following holds.
Assume that $\bar f_{\ba}(\bb)=0$ and that $det(J_{\bar f_{\ba}}(\bb))\neq 0$ (see Notation \ref{Jac}). 
Then  there are  neighbourhoods $O_{\ba}\subset K^\ell$ of $\ba$ (respectively $O_{b_i}\subset K$, $1\leq i\leq m$, of $b_i$) and C$^1$-functions $g_{i}(\x): O_{\ba} \to O_{b_i}$, $1\leq i\leq m$, such that  setting $\bar g:=(g_1,\ldots, g_m)$ and $\bar \ell:=(1,\ldots,\ell)$,
\begin{align} \label{eq9} & \bar g(\ba)=\bb\;\wedge\\
\label{eq10} &\forall\; \x\in O_{\ba}\;\;\big(\bar f(\x,\bar g(\x))=0\; \wedge J_{\bar g}(\x)=-(\nabla_{m}\bar f(\x,\bar g(\x))^{-1}\nabla_{\bar \ell} \bar f(\x,\bar g(\x)))\big)\;\wedge\\
\label{eq11} & \forall\; \x\in O_{\ba}\;\forall\; \y\in O_{\bb}\;(\bar f(\x,\y)=0 \leftrightarrow \y=\bar g(\x)).
\end{align}
We denote by $(\IFTA)_{e}$ the corresponding scheme when the functions $f_1(\x,\y),\ldots, f_m(\x,\y)$ are those induced by the exponential polynomials (with coefficients in $K$).
\edfn

\nota \label{IFTA} As noted in \cite[4.3]{W}, when the topology on $K$ is definable, this implies that whenever the functions $f_i$ are definable (respectively C$^{\infty}$), the $g_i$'s are definable (respectively C$^{\infty}$), using the above equations (\ref{eq9}), (\ref{eq10}), (\ref{eq11}).
If, when the functions $f_i$, $1\leq i\leq m$, are analytic functions in a neighbourhood of $(\ba,\bb)$, the functions $\bar g$ in the scheme $(\IFTA)$ are also analytic in a neighbourhood of $\ba$, we will denote the corresponding scheme $(\IFTA)_{an}$.
\enota

\nota \cite[below Notation 4.6]{W}
Keeping the same notation as in Notation \ref{neigh} and Definition \ref{imp}, and assuming that $\K$ satisfies $(\IFTA)$, we may define 
a map 
\[\hat{}: \Ge^n(\Se_{(\ba,\bb)})^-\to \Ge^{\ell}(\Se_{\ba})^-: f\mapsto \hat{f}
\]
 sending the function $f\restriction O_{\ba}\times O_{\bb} \to K$, where $O_{\bb}:=O_{b_1}\times\ldots\times O_{b_m}$, to the function $\hat f: O_\ba \to K:\x\mapsto f(\x, g_{1}(\x),\ldots,g_{m}(\x))$.
\par It is convenient to introduce an $(\ell+m)$-tuple $(\overline{ \tilde g})$ of functions defined as follows: $\tilde g_i(\x)=x_i$ for $1\leq i\leq \ell$ and $\tilde g_{\ell+i}:=g_i(\x)$, $1\leq i\leq m$. With this notation $\hat f(\x)=f(\overline{ \tilde g}(\x))$. 
\enota  

\lem\label{hat} Suppose $\K$ satisfies $(\IFTA)_{e}$. Let $(\ba,\bb)\in K^{\ell+m}$ and let $f_1,\ldots,f_m, h\in \Ge^{\ell+m}(\Se_{(\ba,\bb)})$.
 Let $\bar f:=(f_{1},\ldots, f_{m})$ and assume that $\bar f(\ba,\bb)=0\;\&\;det(J_{\bar f_{\ba}}(\bb))\neq 0$. 
Then, keeping the same notations as above, the sequence of vectors $\nabla \bar f(\ba,\bb),\;\nabla h(\ba,\bb)$ is $K\!$-linearly independent iff $\nabla \hat{h}(\ba)\neq 0$.
\elem
\pr The proof is the same as the one of \cite[Lemma 4.7]{W} (and it was also used in \cite{M} (see \cite[Lemma 5.1.3]{M})). \qed

\medskip
\par We will need the following {\it lack of flat functions} ($\LFF$) property 
\cite[Lemma 4.5]{W}, \cite[Lemma 25]{S}. 
\dfn \label{lff} We say that $\K$ satisfies $(\LFF)$ if the following holds. 
\par (1) The map sending an element of $\K[\X]^E$ to the corresponding function in $\K[\x]^E$ (possibly partially defined, is injective (see Notation \ref{inj}),
\par (2) Let $\Se$ be a neighbourhood system in $K^n$ and let $M$ be a subring of $\Ge^n(\Se)$
closed under differentiation. Let $I\subset M$ be a finitely generated ideal closed under differentiation and let $[g_1,U_1],\ldots, [g_s,U_s]$ be generators for $I$. 
Let $Z$ be the set of common zeroes of $g_i$, $1\leq i\leq s$, in $U_1\cap\ldots\cap U_s$. Then there is $U\in \Se$ such that $U\cap Z$ is an open subset of $K^n$.
\edfn
\par In case $K$ is either real-closed or an ordered field which is definably complete, then $(\LFF)$ holds in general. Indeed, in that last case, it follows from the following property of solutions of systems of linear differential equations \cite[Theorem 8]{S}, \cite[Proof of Lemma 4.5]{W}:\\
given a non-empty open interval $U$ of $K$ and a system of linear differential equations with coefficients in the ring $R$ of $C^1$-functions from $U$ to $K$, there is a unique solution of that system in $R$.

\

\par If in $\K$, $(\IFTA)_{an}$ holds and if we restrict $M$ to be a subring of $\Ge_{an}^n(\Se)$, then the corresponding property (2) above holds.

\

Observe that given a finite given number of elements of $K[\x]^E$, we can put them in a noetherian differential subring of $K[\x]^E$. Indeed, using the complexity function $\ord$ defined in $K[\X]^E$, this is always possible to find such a ring. An exponential polynomial corresponds to an $\L_{E}$-term and those are constructed by induction in finitely many steps. So we place ourselves in the ordinary polynomial ring (over $K$) generated by all the (finitely many) sub-terms appearing in the construction and their derivatives. This subring of $K[\x]^E$ is closed under differentiation and noetherian. So in such a ring we may apply property $(\LFF)$.
W.l.o.g.  we may place ourselves in a subring of $K_0[\x]^E$, where $K_0$ is a finitely generated subfield of $K$. 
\par The next result was first observed for $(\bar \IR,exp)$ by A. Wilkie \cite{W} but note that it also holds without the assumption of noetherianity, for definably complete structures by G. Jones and A. Wilkie \cite{JW}. Then it was re-used in \cite[Proposition 5.1.4]{M} in the case of the valued field $(\IQ_{p},E_{p})$; there one applied the version $(\IFTA)_{an}$ of the implicit function theorem for analytic functions. 

\prop \label{max} \cite[Theorem 4.9]{W} Assume that $\K$ satisfies $(\IFTA)_{e}$ and $(\LFF)$. 
\par Let $\ra\in K^n$ and let $R_{n}$ be a noetherian subring of $\Ge^n(\Se_{\ra})$ closed under differentiation.
Let $f_{1},\ldots, f_{m}\in R_{n}$, $1\leq m\leq n$, and assume $\ra\in V_n^{reg}(f_{1},\ldots, f_{m})$. Then, exactly one of the following is true:
\begin{enumerate} 
\item[(a)] $n=m$; or,
\item[(b)] $m<n$ and for all $h\in R_{n}$ with $h(\ra) = 0$, h vanishes on $U\cap V_n^{reg}(f_{1},\ldots, f_{m})$ for some open neighbourhood $U$ containing $\ra$,
\item[(c)] $m<n$ and for some $h\in R_{n}$, $\ra\in V_n^{reg}(f_{1},\ldots, f_{m},h)$.
\end{enumerate}\qed
\eprop
\par Note that later, we will not appeal to the above proposition directly, but instead we will go over the main lines of its proof. 

\
\par We will end the section by showing, that in case $\K$ satisfies $(\IFTA)_{e}$, that we can find $\Ecl$-independent, $K$-small elements in an elementary extension of $\K$, using the following immediate consequence of $(\IFTA)_{e}$.
\begin{remark}\label{zero-iso} 
Let $\K$ satisfying $(\IFTA)_{e}$. Let $f_{1},\ldots, f_{m}\in K[\x,\y]^E$, and $(\ba, \bb)\in K^{\ell+m}$, with $\vert \x\vert=\ell, \vert \y\vert=m$. Let $\ba\in K^{\ell}$, $\bb\in K^m$. Assume that $H_{\bar f_{\ba}}(\bb)$ holds, namely $\bar f(\ba,\bb)=0$ and $det(J_{\bar f_{\ba}}(\bb))\neq 0$ (see Definition \ref{Kho}). Then,
 $\bb$ is an isolated zero of the system $\bar f_{\ba}(\y)=0$. 
 \par Note that if the field $K$ is in addition separable as a topological space, then $\Ecl^K(B)$ is countable whenever $B$ is countable. (Indeed the number of Khovanskii systems with coefficients in $B$ is countable and each such system has at most countably many solutions since each solution is isolated.)
\end{remark}

\begin{lemma}\label{sat-ecl-petit} Let $\K$ satisfy $(\IFTA)_{e}$. Let $\K_1$ be a $\vert K\vert^+$-saturated elementary extension of $\K$. Then  there is an element $t\in K_1\setminus \Ecl^{K_{1}}(K)$ with $t\sim_K 0$. More generally for every $n\in \IN^*$ there are $n$ elements $t_1,\ldots, t_n\in K_1$ $\Ecl$-independent over $K$ and $K$-small.
\end{lemma}
\begin{proof} Consider the partial type $tp_K(x)$ consisting of $\L(K)$-formulas expressing that $x\sim _{K} 0$ and $x\notin ecl(K)$. The first property is expressed by the set of formulas $\chi(x,\bar a)$, where $\bar a$ varies in $K$ and the second property by $\neg \exists \bar y\;H_{\bar f}(x,\bar y)$ where $\bar f$ varies in $K[X,\bar Y]^E$. 
By Remark \ref{zero-iso}, this set of formulas is finitely satisfiable. 
So $tp_K(x)$ is realized in a $\vert K\vert^+$-saturated extension of $K$ (see for instance \cite[Theorem 4.3.12]{Marker}).

Then by induction on $n$, assume we found $n$ elements $t_1,\ldots, t_n$ $\Ecl$-independent over $K$ and $K$-small. 
Consider the partial type $tp_{K(t_1,\ldots,t_n)}(x)$ consisting of $\L(K(t_1,\ldots,t_n))$-formulas expressing that $x\sim _{K} 0$ and $x\notin \Ecl^{K_1}(K(t_1,\ldots,t_n))$. Again  by Remark \ref{zero-iso}, it is finitely satisfiable and so it is realized in $\K_1$ by an element $t_{n+1}$ such that $t_1,\dots, t_{n+1}$ are $\Ecl$-independent over $K$ and $K$-small.
\end{proof}

\begin{proposition}\label{generegalter} Let $\K$ satisfy $(\IFTA)_{e}$.  Let $\bar f=(f_1,\ldots,f_m)\subseteq K[\X]^E$, $\vert \X\vert=n> m$. Suppose that  there is $\ba\in V_n^{reg}(\bar f)\cap K^n$.

Then there is an elementary $\mathcal L$-extension $\tilde \K$ of  $\K$ and $\bb\in V_n^{reg}(\bar f)\cap \tilde K^n$ with $\bb-\ba\sim_{K}\bar 0$ and $\dim^{\tilde K}(\bb/K)=n-m$.
In particular, $\bb$ is a generic point of $V_n^{reg}(\bar f)\cap \tilde K^n$.
\end{proposition}
\begin{proof}
Let $\ba\in V_{n}^{reg}(\bar f)$, then $\nabla f_{1}(\ba),\ldots,\nabla f_{m}(\ba)$ are linearly independent over $K$.
By permuting the variables $X_{1},\ldots,X_{n}$, assume that $\nabla_{m} f_{1}(\ba),\ldots,\nabla_{m} f_{m}(\ba)$ are $K$-linearly independent (see Notation \ref{gradient}). So we have $det(J_{\bar f_{\ba_{[n-m]}}}(\ba_{[m]}))\neq 0$, with $\ba:=(\ba_{[n-m]},\ba_{[m]})$ (see Notation \ref{subtuple}).  
By  $(\IFTA)_{e}$, there are definable neighbourhoods $O\subseteq K^{n-m}$ of $\ba_{[n-m]}$, $O'\subseteq K^m$ of $\ba_{[m]}$ and definable functions  $g_{1},\ldots, g_{m}$ from $O\to O'$ such that $\ba_{[m]}=g(\ba_{[n-m]})$ and such that for all $\x \in O$, $\bigwedge_{i=1}^m f_{i}(\x, g_{1}(\x),\ldots,g_{m}(\x))=0$.
By  Lemma \ref{sat-ecl-petit}, there is an elementary $\mathcal L_E$-extension $\tilde \K$ of $\K$ containing $n-m$ $K$-small elements 
$t_1,\ldots,t_{n-m}$ which are   $\Ecl$-independent over $K$.
\;\\
Let $\ta_{[n-m]}:=(t_1,\ldots, t_{n-m})$ and $\bb:=\ba_{[n-m]}+\ta_{[n-m]}\in K_{n-m}$. Then $\bb\in \tilde K$ are $\Ecl$-independent over $K$, $\ba-\bb\sim_K 0$ and $\bigwedge_{i=1}^n f_i(\bb, g_1(\bb),\ldots,g_m(\bb))=0$. 
\end{proof}
 
\section{Topological differential exponential fields}

\subsection{Differential fields expansions}\label{sec:ec}
\par Throughout this section, we will place ourselves in the same setting as in subsection \ref{top}; in particular the language $\L$ is a relational expansion of $\L_{E}\cup \{{}^{-1}\}$. 
Again, we assume that the topological $\L$-field $\K$ is endowed with a definable field topology with corresponding formula $\chi$ and that this topology is a $V$-topology. 
\par Let $\L_{\delta}$ be the expansion of $\L$ by an $E$-derivation $\delta$ and given $\K$, let $\K_{\delta}$ denotes the expansion of $\K$ by an $E$-derivation $\delta$. 
\par Given an $\L$-theory $T$ of topological $\L$-fields, we denote 
by  $T_{\delta}$ the theory $T$ together with the axioms of an $E$-derivation (see Definition \ref{E-der}). In particular if $\K\models T$, then $\K_\delta$ is a model of $T_\delta$. 
\par Any $\L_\delta$-term $t(\x)$ with~$\x=(x_1,\ldots,x_n)$ is equivalent, modulo the theory of differential $E$-fields, to an $\L_{\delta}$-term  $t^*(\bar{\delta}^{m_1}(x_1),\ldots,\bar{\delta}^{m_n}(x_n))$ where $t^*$ is an $\L$-term, for some $(m_1,\ldots,m_n)\in \IN^n.$ 
By possibly adding tautological conjunctions like~$\delta^k(x_i)=\delta^k(x_i)$ if needed, we may assume that all the $m_i$'s are equal. We use the following notation $\bar{\delta}^m(\x):=(\x,\delta(\x),\ldots,\delta^m(\x))$, with $\delta^i(\x):=(\delta^i(x_1),\ldots, \delta^i(x_{n}))$, $1\leq i\leq m$. Recall that we associated with any $\L$-term $t^*$ a quantifier-free formula $D_{t^*}$ and its domain of definition.
\par Therefore, we may associate with any quantifier-free $\L_\delta$-formula $\varphi(\x)$ an equivalent $\L_\delta$-formula, modulo the theory of differential $E$-fields, of the form $\varphi^{*,m}(\bar{\delta}^m(\x))$, $m\in \IN$, where $\varphi^{*,m}$ is an $\L$-quantifier-free formula which arises by uniformly replacing every occurrence of~$\delta^m(x_i)$ by a new variable~$x_i^m$ in~$\varphi$ with the following choice for the order of variables $\varphi^{*,m}(\x^0,\ldots,\x^m)$, where $\x^i=(x_1^i,\ldots,x_n^i)$, $0\leq i\leq m$; furthermore since we made the convention that the functions are not everywhere defined, we assume in addition that the formula $\varphi^{*,m}(\x^0,\ldots,\x^m)$ contains for each term $t^*(\x^0,\ldots,\x^m)$ the quantifier-free $\L$-formula $D_t^*(\x^0,\ldots,\x^m)$. Let $\T_\varphi$ be the set of $\L_\delta$-terms occurring in $\varphi$.

\par Furthermore since we are only interested in existentially closed models, we will  add new variables (that we will quantify existentially) and we replace in the formula $\varphi^{*,m}(\x^0,\ldots,\x^m)$, each occurrence of an $\L$-subterm of the form $s^{-1}$ by a new variable $u$ together with the existential formula $\exists u\;u s=1$, in order to transform atomic $\L$-formulas into atomic $\L^-$-formulas in variables $\x^0,\ldots,\x^m, \bar u$. Note that $\delta(u)$ is expressed in terms of $s,\;\delta(s)$.
So we get 
\begin{equation}\label{eq:*}
\varphi(\x) \wedge \bigwedge_{t\in \T_\varphi} D_{t^*}(\bar{\delta}^m(\x)) \leftrightarrow \exists \bar u\,\varphi_-^{*,m}(\bar{\delta}^m(\x),\bar u)\wedge \bigwedge_{t\in \T_\varphi} D_{t^*}(\bar{\delta}^m(\x)),
\end{equation}
where now $\varphi_-^{*,m}$ is a quantifier-free $\L^-$-formula.
We will call the least such $m$, the order of the quantifier-free $\La_{\delta}$-formula $\varphi$. We will call an atomic formula of the form $s(\y)=0$ an $\L^{-}$-equation (or $\L_E$-equation), where $s(\y)$ is an $\L^-$-term.
We will usually drop the superscript $m$ in the formula $\varphi_-^{*,m}$. We will make the following notational simplifications: we will no longer specify that we work on the domains of definitions of our terms.

\subsection{Scheme $(\DL)_{E}$}
Given a model-complete theory $T$ of topological $\L$-fields, we consider the class of existentially closed differential expansions of models of $T$ and under additional assumptions on the class of models of $T$, we will show that this class is elementary and we produce an axiomatisation. Namely, by a scheme of first-order axioms, we will express that certain systems of differential exponential equations have a solution. In order to determine which ones, we first associate, using the process explained above, to a quantifier-free $\L_\delta$-formula $\varphi(\x)$ of order $m$, a quantifier-free $\L^{-}$-formula  $\varphi_-^{*,m}(\bar{\delta}^m(\x),\bar u)$. 
From now on, we will make the additional hypothesis that $\varphi(\x)$ is a finite conjunction of basic formulas (namely either an atomic formula or the negation of an atomic formula), and one can easily check that the associated formula $\varphi_-^{*,m}(\bar{\delta}^m(\x),\bar u)$ is also a finite conjunction of basic formulas.
Since the derivation extends in a unique way to the $\Ecl$-closure, we enumerate partitions of the variables into two subsets: a first one where we impose no conditions and the other one where we express that there are regular solutions of an $E$-variety over this first subset of variables. 
\medskip

\dfn \label{const} Let $\K_{\delta}$ be a differential topological $\L$-field. 
\par We start with a (finite) system $E$-polynomial equations $(\x^{0},\ldots,\x^{m})\in V$ over $K$ where $V$ denotes the corresponding $E$-variety and $\x^{i}:=(x_{1}^i,\ldots, x_{n}^i),$ $1\leq i\leq m$.
\par Let $n\geq \ell_0\geq \ell_{1}\geq\ldots\geq \ell_{m-1}\geq 0$ and set $\x^i_{[\ell_i]}:=(x_{1}^i,\ldots, x_{\ell_i}^i)$, $0\leq i\leq m-1$. Let $\z_{j,i}$ be a tuple of new variables, $\ell_i+1\leq j\leq n$, $0\leq i\leq m-1$ and let $\bar z:=(\z_{(\ell_{i}+1),i},\ldots, \z_{n,i})_{0\leq i\leq m-1}$. Set $\bar \ell:=(\ell_0, \ell_{1},\ldots, \ell_{m-1})$.
\par We then define a Khovanskii formula $H_{\bar \ell}(\x^{0},\ldots,\x^{m-1},\bar z)$, as 
\[
H_{\bar \ell}(\x^{0},\ldots,\x^{m-1},\bar z):= \bigwedge_{i=0}^{m-1} \bigwedge_{j=\ell_i+1}^n H_{\bar f_{j,i}}(x_{j}^i,\z_{j,i}),
\]
where $H_{\bar f_{j,i}}(x_{j}^i,\z_{j,i})$ are Khovanskii systems (see Definition \ref{Kho})
expressing that each element $x_j^i$, $\ell_i+1\leq j\leq n$, of the subtuple $(x_{\ell_i+1}^i,\ldots, x_{n}^i)$ of $\x^i$ is in the $\Ecl^K$-closure of $\x^0_{[\ell_0]},\ldots, \x^i_{[\ell_i]}$, and
$\bar f_{j,i}$ is a tuple of $E$-polynomials 
with coefficients in $K(\x_{[\ell_{i}]}^i,\ldots, \x_{[\ell_{0 }]}^0)$, $0\leq i\leq m-1$.
\par Then a prepared system $(\x^0,\x^1_{[\ell_0]},\ldots, \x^{m}_{[\ell_{m-1}]},\bar z)\in V_H^{prep}$ (with corresponding Khovanskii formula $H_{\bar \ell}$) is 
a conjunction of basic $\L^-(K)$-formulas consisting of:
\begin{enumerate}
\item a Khovanskii formula $H_{\bar \ell}(\x^{0},\ldots,\x^{m-1},\bar z)$ (keeping the above notation),
 \item a conjunction of $E$-polynomial equations $(\x^0,\x^1_{[\ell_0]},\ldots, \x^{m}_{[\ell_{m-1}]}, \bar z)\in V_H$ obtained from the initial system $(\x^{0},\ldots,\x^{m})\in V$ and the Khovanskii systems $H_{\bar f_{j,i}}(x_{j}^i,\z_{j,i})$, by replacing every tuple of variables 
 $(x_{j}^{k+1},\z_{j,k+1})$, $0\leq k \leq m-1$, $\ell_{k}+1\leq j\leq n$, by the tuple $\ta^{1,*}_{\bar f_{j,k+1}}$ of $E$-rational functions (over $K$) in $x_{j}^{k}, \z_{j,k}, \x_{[\ell_k]}^{k},\ldots,\x_{[\ell_{1}]}^1, \x_{[\ell_0]}^0$ introduced in Notation \ref{rat}, obtained by clearing denominators and
 \item a conjunction of negation of atomic formulas saying that the denominators of the terms $\ta_{\bar f_{j,k+1}}^{1,*}$ are non-zero.
\end{enumerate}
\par So a prepared system consists of the conjunction of two conditions of the form: belonging to some $E$-variety and to the complement of such definable set that we will call an $E$-Zariski open set. As usual  we denote the set of solutions in $K$ of the formula $(\x^0,\x^1_{[\ell_0]},\ldots, \x^{m}_{[\ell_{m-1}]},\bar z)\in V_H^{prep}$ by $V_H^{prep}(K)$.
 \par We allow the case where there is no Khovankii formula $H_{\bar \ell}$ and in this case the formula $(\x^{0},\ldots,\x^{m})\in V_{\emptyset}^{prep}$ (with empty tuple $\bar z$) is equal to $(\x^{0},\ldots,\x^{m})\in V$.
\par We let $K_0$ be the subfield of $K$ generated by the elements of $K$ occurring in the formula $(\x^0,\x^1_{[\ell_0]},\ldots, \x^{m}_{[\ell_{m-1}]},\bar z)\in V_H^{prep}$. We will say that $K_{0}$ is associated with $(V, H)$.
\par Let $\bar f$ be a tuple of $E$-polynomials and let $O$ be the ($E$-Zariski) open set such that $V_H^{prep}(K)=V(\bar f)\cap O$. 
Let $N:=n+\sum_{j=0}^{m-2} \ell_j+\vert\bar z\vert$.
\par Then we say that the system $(\x^0,\x^1_{[\ell_0]},\ldots, \x^{m}_{[\ell_{m-1}]},\bar z)\in V_H^{prep}$ is well prepared over $K_0$ if there is a tuple of $E$-polynomials $\bar h(\x^0,\x^1_{[\ell_0]},\ldots, \x^{m}_{[\ell_{m-1}]},\bar z)$ with coefficients in $K_0$ and if there are (non-empty) tuple of indices $\bar i:=(i_{1},\ldots,i_{p})$, $1\leq i_{1}<\ldots<i_{p}\leq \ell_{m-1}$ among the indices of the tuple $\x^{m}_{[\ell_{m-1}]}$ and another tuple of indices $\bar j:=(j_{1},\ldots,j_{r})$ among the indices of the tuple $(\x^0,\x^1_{[\ell_0]},\ldots, \x^{m-1}_{[\ell_{m-2}]},\bar z)$, $1\leq j_{1}<\ldots<j_{r}<N$ such that  $V_{N+\ell_{m-1},(\bar j,\bar i)}^{reg}(\bar h)\cap O=V^{prep}_H(K)$.
\edfn
\par From now on we will not insist on which are the variables occurring in the formula $(\x^0,\x^1_{[\ell_0]},\ldots, \x^{m}_{[\ell_{m-1}]},\bar z)\in V_H^{prep}$ and we will simply denote it by $(\x^0,\x^1,\ldots, \x^{m},\bar z)\in V_H^{prep}$.

\dfn\label{DL} The scheme $(\DL)_E$ has the following form: 
for each (finite) system $E$-polynomial equations $(\x^{0},\ldots,\x^{m})\in V_n$ over $K$ and associated prepared system \\$(\x^0,\x^1,\ldots, \x^{m},\bar z)\in V_H^{prep}$ with Khovanskii formula $H_{\bar \ell}(\x^{0},\ldots,\x^{m-1},\bar z)$, with $\bar \ell:=(\ell_0,\ldots,\ell_{m-1})$, $n\geq \ell_0\geq \ldots\geq \ell_{m-1}$, letting $O$ be an $E$-Zariski open set such that $V^{prep}_H(K)=V_{n}\cap O$, and further assuming it is well-prepared by a tuple of $E$-polynomials $\bar h$ with the property that $V_{N+\ell_{m-1},(\bar j,\bar i)}^{reg}(\bar h)\cap O=V^{prep}_H(K)$ with $N:=n+\sum_{j=0}^{m-2} \ell_j+\vert\bar z\vert$, $\bar i:=(i_{1},\ldots,i_{p})$, $1\leq i_{1}<\ldots<i_{p}\leq \ell_{m-1}$ among the indices of the tuple $\x^{m}_{[\ell_{m-1}]}$ and another tuple of indices $\bar j:=(j_{1},\ldots,j_{r})$ among the indices of the tuple $(\x^0,\x^1_{[\ell_0]},\ldots, \x^{m-1}_{[\ell_{m-2}]},\bar z)$, $1\leq j_{1}<\ldots<j_{r}<N$, we have:
\begin{align*}
\forall \bar d\;\forall \x^0\;\ldots\;\forall \x^m\;&\;\big(\exists \bar z\;\big ((\x^{0},\ldots,\x^{m},\bar z)\in V_{\bar i}^{reg}(h_1,\ldots,h_p))\rightarrow\\
&(\exists \alpha\;\exists \bar w\;\; (\bar\delta^m(\alpha),\bar w)\in V_{\bar i}^{reg}(h_1,\ldots,h_p)\,\wedge\;\chi((\bar\delta^m(\alpha),\bar w)-(\x^0,\ldots,\x^m,\bar z),\bar d))\big).
\end{align*}
\edfn
\par Note that by quantifying over the coefficients of $h_1,\ldots,h_p$, this scheme is first-order.
\rem \label{dense} In a model $\K_\delta\models T_\delta$ of the scheme $(\DL)_E$, the differential points are dense in all cartesian products of $K$. Let $O\subset K^{m+1}$ and $(a_0,\ldots,a_{m})\in O$. Consider the system $x_m=a_m$, then it is well-prepared since taking $h:=x_m-a_m$, $(\partial_{x_0} h,\ldots,\partial_{x_m} h)=(0,\ldots,0,1)$ and so the tuple $(a_0,\ldots,a_{m-1},a_m)\in V^{reg}(h).$
So we find a differential solution $b$ such that $\delta^{m}(b)=a_m$ and $\bar \delta^{m-1}(b)$ is close to $(a_{0},\ldots,a_{m-1})$. This is analogous to \cite[Lemma 3.12]{GP}.
\par The same argument shows that the subfield of constants $C_{K}$ is dense in $K$ (and recall that since $\delta$ is an $E$-derivation, $C_{K}$ is a partial $E$-subfield of $K$ which is relatively algebraically closed in $K$). We even have that $\Ecl^K(C_K)=C_K$ by Lemmas \ref{der-multi}, \ref{der}.
\erem
\medskip
\par Recall that we always work with topological fields endowed with a definable $V$-topology with corresponding formula $\chi$. The main result of this section is:
\thm \label{ec}
Let $T$ be a model-complete theory of topological $\L$-fields.
Assume that $\K\models T$ and that the differential expansion $\K_{\delta}$ is a model of $T_\delta\cup(DL)_{E}$. Then $\K_{\delta}$ is existentially closed in the class of  models of $T_{\delta}$. In particular if the theory $T_\delta\cup (DL)_E$ is consistent, then it is model-complete.
Assume that the models of $T$ satisfy either the schemes $(\IFTA)_{e}$ and $(\LFF)$ or satisfy $(\IFTA)_{an}$, then the theory $T_\delta\cup (DL)_E$ is consistent and so it is the model-companion of the theory $T_\delta$.
\ethm
\par The above theorem will follow from Theorems \ref{emb_delta} and \ref{DLE}.
\par The strategy of the proof is the following. First show that a model $\K_\delta\models T_\delta$ 
satisfying $(\IFTA)_{e}$ and $(\LFF)$ can be embedded in $\tilde K_\delta\models T_\delta\cup (\DL)_{E}$ (Theorem \ref{emb_delta}). 
Second show that if $T_{\delta}\cup (\DL)_{E}$ is consistent, then it axiomatizes the existentially closed models of $T_\delta$ (Theorem \ref{DLE}). 

\

We begin by realizing one instance of the scheme $(\DL)_E$ in a differential extension of $\K_\delta$. 
\medskip
\begin{lemma}\label{iter1} Let $\K_\delta \models T_\delta$ and  
suppose $\K$ satisfies $(\IFTA)_{e}$. Let $\M$ be a $\vert K\vert^+$-saturated elementary $\L$-extension of $\K$.
Consider a (finite) system $E$-polynomial equations $(\x^{0},\ldots,\x^{m})\in V_n$ over $K$, a Khovanskii formula $H_{\bar \ell}(\x^{0},\ldots,\x^{m-1},\bar z)$ with $\bar \ell=(\ell_0,\ldots,\ell_{m-1})$, $n\geq \ell_0\geq \ldots\geq \ell_{m-1}$, and the associated prepared system $(\x^0,\x^1,\ldots, \x^{m},\bar z)\in V_H^{prep}$.
Assume it is well-prepared by $\bar h:=(h_1,\ldots,h_p)$, $1\leq p\leq \ell_{m-1}$, over a finitely generated subfield $K_0$ of $K$. Fix $\bar d\in K$.
Assume that for some $\ba:=(\ba^0\;\ldots\;, \ba^m)\in K$, $\vert \ba^i\vert=n$, $0\leq i\leq m$, we have for some $\bar u\in K$, $\vert \bar u\vert=\vert \bar z\vert$, that
\[
 \K\models (\ba^{0},\ldots, \ba^{m},\bar u)\in V_{\bar i}^{reg}(h_1,\ldots,h_p).
 \]
\par Then, we can find a tuple of elements $\bar \gamma\in M$ and we can extend $\delta$ on $\Ecl^{M}(K,\bar \gamma)$ (and then to $M$) such that for some $(\bar \delta^{m}(\alpha),\overline{\tilde u})\in \Ecl^{M}(K,\bar \gamma)$, we have:
\[
\Ecl^{M}(K,\bar \gamma)\models (\bar \delta^{m}(\alpha),\overline{\tilde u})\in V_{\bar i}^{reg}(h_1,\ldots,h_p)\;\wedge\;\bar \chi((\bar \delta^{m}(\alpha),\overline{\tilde u})-(\ba, \bar u),\bar d).
\]
\end{lemma}
\begin{proof} 
First let us observe that the saturation hypothesis on $M$ will only be used in order to find $K$-small elements which are  $\Ecl$-independent over $K$.
\par For ease of notation, suppose first that $m=1$. Suppose the Khovanskii formula $H_{\bar \ell}$ is of the form $\bigwedge_{i=1}^{n-\ell} H_i(\x_{\ell+i}^0,\bold z_{\ell+i})$, $0<\ell<n$.
Let $\ba:=(\ba^0,\ba^1)$ and let $\bar u:=(\bu_{\ell+1},\ldots,\bu_{n})\in K$ be such that 
\[
\K\models \bigwedge_{i=1}^{n-\ell} H_i(a_{\ell+i}^0,\bu_{\ell+i}).
\]
Let $n_i:=\vert \bu_{\ell+i}\vert$, $1\leq i\leq n-\ell$ and $N$ be the length of $(\ba^0, \bar u)$.
\par We have assumed that $V_H^{prep}$ is well-prepared over $K_0$ and so there is a tuple of $E$-polynomials $\bar h:=(h_1(\x^0,\bar z, \x^1_{[\ell]}),\ldots, h_k(\x^0,\bar z,\x^1_{[\ell]}))$ with $\nabla h_{1}(\ba^0,\bar u,\ba^1_{[\ell]}),\ldots, \nabla h_{k}(\ba^0,\bar u,\ba^1_\ell)$ are $K$-linearly independent, with $k<N+\ell$. 
Furthermore, for some tuple of increasing indices $\bar i:=(i_{1},\ldots,i_{p})$, $1\leq i_{1}<\ldots<i_{p}\leq \ell$, we decompose $\ba^1_{[\ell]}$ accordingly into two subtuples: $\ba^1_{[p]}$ and $\ba^1_{[\ell-p]}$ (see Notation \ref{subtuple}), allowing if $p=\ell$, $\ba^1_{[\ell-p]}$ to be an empty tuple, and we get that the determinant of the matrix $(\nabla_{\bar i}h_{1\,(\ba^0,\bar u, \ba^1_{[p-\ell]})}(\ba^1_{[p]}),\ldots,\nabla_{\bar i}h_{p\,(\ba^0,\bar u, \ba^1_{[p-\ell]})}(\ba^1_{[p]}))$ is nonzero.
Assume that $(\ba^0,\bar u,\ba^1_{[\ell]})\in V_{\bar i}^{reg}(h_1,\ldots,h_p)$. 
\par Decompose the tuple of variables $(v_1,\ldots, v_\ell)$ as $(v_{[\ell-p]}, v_{i_1},\ldots,v_{i_p})$ and set $\bar y=(\x^0,\bar z,v_{[\ell-p]})$.
 Then we apply directly hypothesis $(\IFTA)_{e}$. There exist $O_{1}$ a definable neighbourhood of $(\ba^0, \bar u, \ba^1_{[p-\ell]})$ and $O_{2}$ a definable neighbourhood of $\ba_{[p]}^1$ and definable $C^\infty$ functions $g_{i}$ from $O_{1}$ to $O_{2}$, $1\leq i\leq p$, such that 
\[ \bigwedge_{i=1}^p g_{i}(\ba^0,\bar u,\ba^1_{[\ell-p]})=a_{i}^1 \;\;\wedge\;\;\forall\; \bar y\in O_{1}\;\;
(\bigwedge_{i=1}^p h_{i}(\bar y, g_{1}(\bar y),\ldots,g_{p}(\bar y))=0).
\]

\par Recall that we put the product topology on $M^{\tilde N}$ with $\tilde N=N+(\ell-p)$ (the length of $(\ba^0,\bar u,\ba^1_{[\ell-p]})$). Let $\pi$ be the projection sending a tuple $(\ba^0,\bu,\ba^1_{[p-\ell]})\in M^{\tilde N}$ to the subtuple $\ba_{[\ell]}^0\in M^{\ell}$ and $\pi_i$ the projection sending $(\ba^0,\bu,\ba^1_{[\ell-p]})$ to the subtuple $(a_{\ell+i}^0, \bu_{\ell+i})\in M^{n_i+1}$, $1\leq i\leq n-\ell$.
 
Let $(a_{\ell+i}^0, \bu_{\ell+i})$ be regular zeroes of  each system $H_i(x_{\ell+i}^0,\z_{\ell+i})$, $1\leq i\leq n-\ell$, over $\IQ(\bar c, \ba_{[\ell]}^0)$, where $\bar c$ is the tuple of coefficients (in $K$) of the $E$-polynomials occurring in $H_i$. For each $1\leq i\leq n-\ell$, we apply $(\IFTA)_{e}$ in $\M$ and find a neighbourhood $O_{1,1}$ of $\ba_{[\ell]}^0$ with $O_{1,1}\subset \pi(O_{1})$ and a neighbourhood $O_{1,\ell+i}$ of $(a_{\ell+i}^0, \bu_{\ell+i})$ with $O_{1,\ell+i}\subset \pi_{i}(O_{1})$ and definable functions $f_{i,1}, \ldots, f_{i,n_{i}}$ from $O_{1,1}$ to $O_{1,\ell+i}$ such that 
\begin{equation}\label{Kh}
\bigwedge_{i=1}^{n-\ell} f_{i,1}(\ba_{[\ell]}^0)=a_{\ell+i}^0 \wedge \bigwedge_{j=1}^{n_{i}} f_{i,j}(\ba_{[\ell]}^0)=u_{\ell+i,j}\wedge
\forall\; \bar y\in O_{1,1}\;
(\bigwedge_{i=1}^{n-\ell} H_{i}(f_{i,1}(\bar y),\ldots,f_{i,n_{i}}(\bar y))).
\end{equation}
Let $\bar f_{i}:=(f_{i,1}(\bar w),\ldots,f_{i,n_{i}}(\bar w))$ with $\bar w=(w_1,\ldots,w_\ell)$.  Applying $\bar f_{i}$ to $(\ba_{[\ell]}^0+\ta_{[\ell]})$,
we get a solution to each system $H_i(x_{\ell+i}^0,\z_{\ell+i})$, close to $(a_{\ell+i}^0,\bu_{\ell+i})$, $1\leq i\leq n-\ell$.
Denote this solution by $(a_{\ell+i}', \bu_{\ell+i}')$, $1\leq i\leq n-\ell$. 
Let 
\[(\widetilde{\ba},\widetilde{\bu}):= (\ba_{[\ell]}^0+\ta_{[\ell]}, \ba^1_{[\ell-p]}, a_{\ell+1}',\ldots, a_{n}', \bu_{\ell+1}',\ldots,\bu'_{n}).
\] 
Since $(\widetilde{\ba},\widetilde{\bu})$ belongs to $O_{1}(M)$, we may apply the functions $g_{1},\ldots,g_{p}$ in order to obtain $(g_{1}(\widetilde{\ba},\widetilde{\bu}),\ldots,g_{p}(\widetilde{\ba},\widetilde{\bu}))\in V_{p}(\bar h_{\widetilde {\ba},\widetilde{\bu}})$. Set $(\tilde b_1,\ldots,\tilde b_\ell):=(g_{1}(\widetilde{\ba},\widetilde{\bu}),\ldots,g_{p}(\widetilde{\ba},\widetilde{\bu}), \ba^1_{[\ell-p]})$.

\par Since now $a_1^0+t_1,\ldots,a_\ell^0+t_\ell$ are $\Ecl^M$-independent over $K$, we may define 
\begin{equation}\label{star}
\delta(a_1^0+t_1):=\tilde b_1,\ldots,\delta(a_\ell^0+t_\ell):=\tilde b_\ell.
\end{equation}
 Note that the values of the successive derivatives of $\tilde b_{1},\ldots, \tilde b_{\ell}$ are determined since the last part of the tuple, namely $\ba^1_{[\ell-p]}$, is in $K$ and we can express $\delta(\tilde b_{1}),\ldots, \delta(\tilde b_{p})$ using that $(\tilde b_{1},\ldots, \tilde b_{p})$ is a regular zero of $V_{p}(\bar h_{\widetilde {\ba},\widetilde{\bu}})$. Note that $\tilde b_{1},\ldots, \tilde b_{\ell}\in \Ecl^{M}(K,\ba_{[\ell]}^0+\ta_{[\ell]}).$
By equation (\ref{Kh}), $a_{\ell+1}',\ldots, a_{n}'\in \Ecl^M(\bar c, \ba_{[\ell]}^0+\ta_{[\ell]})$, we can also express their derivatives in terms of $\ba_{[\ell]}^0+\ta_{[\ell]}, a_{\ell+1}',\ldots, a_{n}'$, the witnesses $\bu_{\ell+1}',\ldots,\bu_{n}'$ and the derivatives of $\ba_{[\ell]}^0+\ta_{[\ell]}$, namely $\tilde b_{1},\ldots, \tilde b_{\ell}$.
So first we extend $\delta$ on $\Ecl^{M}(K,\ba_{[\ell]}^0+\ta_{[\ell]})$ sending the tuple $\ba_{[\ell]}^0+\ta_{[\ell]}$ to  $(\tilde b_1,\ldots,\tilde b_\ell)$ and then by Corollary \ref{der-ext-gener} to $M$. This extension is uniquely determined on the subfield of $M$ generated by $K$, $\ba_{[\ell]}^0+\ta_{[\ell]},a_{\ell+1}',\ldots, a_{n}', \bu_{\ell+1}',\ldots,\bu_{n}'$ and $\tilde b_1,\ldots,\tilde b_\ell$ (see Lemma \ref{der}).
\par Assume now that $m>1$. Then we proceed as before, replacing in the above discussion $\ba_{[\ell]}^1$ by $\ba_{[\ell_{m-1}]}^m$ and $\ba_{\ell}^0$ by $\ba_{\ell_0}^0, \ba_{\ell_1}^1, \ldots, \ba_{\ell_{m-1}}^{m-1}$,
\end{proof}

\thm \label{emb_delta}
Let $\K\models T$ and
suppose $\K$ satisfies $(\IFTA)_{e}$.
Then the differential expansion $\K_{\delta}$ can be embedded in a model $\tilde \K_{\delta}$ of $T_\delta\cup (\DL)_{E}$. 
\ethm
\pr We adapt \cite[Lemma 3.7]{GP} and \cite[Proposition 3.9]{GP} to this exponential setting. 
The differential extension $\tilde \K_{\delta}$ will be built as the union of a chain of differential extensions of $\K_{\delta}$ which will be in addition $\L$-elementary extensions of $\K$. In particular, we get that $\tilde \K$ is an $\L$-elementary extension of $\K$. We first construct such extension $\tilde \K_\delta$ where all the instances of the scheme $(\DL)_E$ with coefficients in $K$ are satisfied using transfinite induction and then we repeat the construction replacing in the previous argument $\K_\delta$ by $\tilde \K_\delta$ and we do it $\omega$ times. The union of this chain of elementary extensions will be a model of the scheme $(\DL)_E$ and an elementary extension of $\K$.

It suffices to show that given an instance of the scheme $(\DL)_E$, we can find an $\L_\delta$-elementary extension $\K_1$ of $\K_{\delta}$.
This was achieved in  Lemma \ref{iter1}.
\qed
\medskip
\par Recall that $\L$ is a first-order language satisfying the assumptions of section \ref{top}.
\par Now let us show that we can always associate with a quantifier-free $\L_\delta(K)$-formula, a prepared system.
\dfn \label{const2} Let $\K_{\delta}$ be a differential topological $\L$-field.
Let $\varphi(\x)$ be a finite conjunction of basic $\L_\delta(K)$-formulas of order $m$ with $\x:=(x_1,\ldots, x_n)$. We associated with $\varphi(\x)$ a $\L^-(K)$-formula $\varphi_-^{*,m}(\bar \delta^m(\x),\bar u)$ which is of the form $(\bar \delta^m(\x),\bar u)\in V\cap \tilde O$, where $V$ is an $E$-variety and $\tilde O$ a definable open subset (see equation \ref{eq:*}). 
\par Denote by $\x^{i}:=(x_{1}^i,\ldots, x_{n}^i),$ $0\leq i\leq m$ with $\x^{0}=\x=(x_{1},\ldots, x_{n})$. Since the extra-variables $\bar u$ that we added in order to only consider $\L_{-}$-terms, are in the $\L_E$-definable closure of $\x^0$ (and in particular in the $\Ecl^K$-closure) of $\x^0,\ldots, \x^m$, we will assume that they occur within $\x^0$.
\par We associated with $(\x^0,\ldots,\x^m)\in V$ a Khovanskii formula $H_{\bar \ell}(\x^{0},\ldots,\x^{m-1},\bar z)$ with extra variables $\bar z$ and tuple of indices $n\geq \ell_0\geq \ell_{1}\geq\ldots\geq \ell_{m-1}\geq 0$. 
We constructed a formula $(\x^0,\ldots,\x^m,\bar z)\in V_H^{prep}$ such that $V_H^{prep}(K)$ is of the form $V(\bar f)\cap O$ for some tuple of $E$-polynomials $\bar f$ an definable open set $O$ (see Definition \ref{const}). 
\par Set $\varphi^*_H(\x^0,\ldots,\x^m,\bar z)$ to be the formula $(\x^0,\ldots,\x^m,\bar z)\in V_H^{prep}\cap O\cap \tilde O.$ From now on rename $O$ the open set $O\cap \tilde O$.
\edfn

\par Let $\bar c$ be new constant symbols that will be interpreted by the parameters coming from $K$ in the Khovanskii formula.

Note that in case we do have a non-trivial relation between the $\x^i$, $0\leq i\leq m-1$, with coefficients in $\IQ(\bar c)$, they cannot be all $\Ecl$-independent over $\IQ(\bar c)$ by Corollary \ref{norel}.

\

In the next lemma we will show that we may always assume that a prepared system is well-prepared, whenever we look for solutions of that system outside the $\Ecl$-closure of the subfield generated by the coefficients of the $E$-polynomials occurring in the system.
\lem \label{wellprep} Let $\K$ satisfy $(\IFTA)_{e}$ together with $(\LFF)$. Suppose $\K_\delta$ satisfies the scheme $(\DL)_E$. Let $V(\bar f)$ be an $E$-variety over $K$, $H$ an Khovankii formula and let  $V_H^{prep}$ be the corresponding prepared system. Assume that $V_H^{prep}(K)=V(\bar f)\cap O$, where $O$ is an open set. Let $K_0$ be a finitely generated subfield of $K$ associated with $(V, H)$. 
\par Suppose $(\ba^0, \ba^1_{[\ell_0]},\ldots, \ba^m_{[\ell_{m-1}]},  \bar u)\in V_H^{prep}(K)\setminus \Ecl^{K}(K_0)$
and that $\bar f_{\ba^0,\ba^1_{[\ell_0]},\dots,\ba^{m-1}_{[\ell_{m-2}]},\bar u}[\X^m_{[\ell_{m-1}]}]\in {(K[\X^m_{[\ell_{m-1}]}]^E)}^{\vert \bar f\vert}\setminus\{0\}$.
Then there is a differential point in $V(\bar f)$ close to $(\ba^0,\ldots,\ba^m)$.
\elem
\pr We will do the proof in the case $m=1$, for the general case replace $\ba^1_{[\ell_0]}$ by $\ba^m_{[\ell_{m-1}]}$. Set $\ell_0=\ell$.
\par We show that we can find a tuple of $E$-polynomials $\bar h$ and a tuple $(\s_1,\s_2,\ra)$ such that the system is well-prepared and there is a neighbourhood of $(\s_1,\s_2,\ra)$ included in $V_H^{prep}(K)$.  
Let $\bar f_{\x^0,\bar z,\x^1}$ be a tuple of $E$-polynomials $\bar f_{\x^0,\bar z,\x^1}$ such that 
$$V_H^{prep}:=\{\x^1_{[\ell]}\in V(\bar f_{\x^0,\bar z})\;\&\;\bigwedge_{i=1}^{n-\ell} H_i(x_{1,\ell+i},\bar z_{\ell+i})\}.$$
\par Let $\s'_1:=(s'_{1,1},\ldots,s'_{1,n})$, $\s'_2:=(\s'_{2,\ell+1},\ldots, \s'_{2,n})$ with $\vert \s'_{2,\ell+i}\vert=\vert \bar z_{\ell+i}\vert$, $1\leq i\leq n-\ell$, $N=\vert \x^0\vert+ \vert \bar z\vert$, let
\begin{align*}
S_{(\ba^0, \bar u, \ba^1_{[\ell]})}&:=\{(\s'_1, \s'_2, \ra')\in K^{N+\ell}: \vert \s'_1\vert=\vert \x^0\vert, \vert \s'_2\vert= \vert \bar z\vert, \vert \ra'\vert=\ell\;\& \\
&\ra'\in V_{\ell}(\bar f_{\s'_1,\s'_2})\;\&\;\bigwedge_{i=1}^{n-\ell} H_i(s'_{1,\ell+i},\s'_{2,\ell+i})\;\&\;(\s'_1,\s'_2,\ra')\notin \Ecl(K_0)\}\cap (\ba^0, \bar u, \ba^1_{[\ell]})+\bar\chi(K,\bar d).
\end{align*}
The set $S_{(\ba^0, \bar u, \ba^1_{[\ell]})}$ is non-empty since it contains $(\ba^0, \bar u, \ba^1_{[\ell]})$. 
Let $R_{N+\ell}$ be a noetherian subring of $K_0[\x^0, \bar z, \x^1_{[\ell]}]^E$ 
closed under differentiation and containing $\bar f(\x^0,\bar z, \x^1_{[\ell]})$ associated to the corresponding the $E$-polynomials.
Let $M$ be the ring of germs in $\Ge^{N+\ell}(\Se_{(\ba^0, \bar u, \ba^1_{[\ell]})})$ induced by the elements of $R_{N+\ell}$.
Denote by $R_{N+\ell}^{<\omega}$ the set of all finite tuples of elements of $R_{N+\ell}$.
\par Let $Ann:=\{(\s'_1, \s'_2, \ra')\in S_{(\ba^0, \bar u, \ba^1_{[\ell]})}\colon$ for some 
$\bar q\in R_{N+\ell}^{<\omega},\;  (\s'_1, \s'_2, \ra')\in V(\bar q)\}$.
\par By assumption the tuple $((\ba^0,\bar u,\ba_{[\ell]}^1),\bar f)$ belongs to $Ann$.
\par Suppose that we have $((\s_{1,n}, \s_{2,n}, \ra_{n}),\bar q_n)\in Ann$, $n\in \IN^*$, with the ideals $\langle \bar q_n\rangle$ forming an increasing chain. By noetherianity of $R_{N+\ell}$,
we can assume such a chain is finite and there is $m_0$ such that for all $m\geq m_0$, $\langle \bar q_{m_0}\rangle=\langle \bar q_m\rangle$, for all $m\geq m_0$.
\par So we may choose among $(\s'_1, \s'_2, \ra')\in S_{(\ba^0, \bar u, \ba^1_{[\ell]})}$, those such that there is $\bar q\in R_{N+\ell}^{<\omega}$ such that $((\s'_1, \s'_2, \ra'),\bar q)\in Ann$ and $\langle \bar q\rangle$ maximal in $R_{N+\ell}$ with, letting $\bar q=(q_1,\ldots,q_k)$: 
\par $\bullet$ first choosing those such that the tuple of indices $\bar i:=(i_{1},\ldots,i_{p})$ with
 $1\leq i_{1}<\ldots<i_{p}\leq \ell$  such that the determinant of the matrix $(\nabla_{\bar i}q_{1\,(\s_1,\s_2)}(\ra),\ldots,\nabla_{\bar i}q_{p\,(\s_1,\s_2)}(\ra))$ is nonzero, is of maximal length, 
\par $\bullet$ second  choosing, among the first ones,  those such that $\nabla q_{1}(\s'_1,\s'_2,\ra'),\ldots, \nabla q_{k}(\s'_1,\s'_2,\ra')$ are $K$-linearly independent with $k$ maximal.
\par Note that by definition of $S_{(\ba^0, \bar u, \ba^1_{[\ell]})}$, we have that $k<N+\ell$. 
\par Note also that the first condition can always be satisfied.  By assumption on $(\ba^0, \bar u, \ba^1_{[\ell]})$, we have that $\bar f_{\ba^0,\ba^1_{[\ell]},\bar u}[\x^1_{[\ell]}]\in {(K[\x^1_{[\ell]}]^E)}^{\vert \bar f\vert}\setminus\{0\}$. So, consider the ideal $J:=\{q\in R_{N+\ell}\colon q_{(s_1,s_2)}(r)=0\}$. If $J$ is closed under all partial derivatives $\partial_{x_i}$, then by Fact \ref{dpnul}, we get a contradiction since $J$ cannot be equal to $K[\x^1_{[\ell]}]^E$.

\

\par Let $\bar j:=(j_{1},\ldots,j_{k-p})$ be a tuple of increasing indices among $\{1,\ldots,N\}$ such that $(\partial_{y_{j_1}} q_{p+1,\ra'}(\s'_1,\s'_2),\ldots, \partial_{y_{j_{k-p}}} q_{k,\ra'}(\s'_1,\s'_2))$ are $K$-linearly independent.
(So the determinant of the matrix $(\nabla_{\bar j} q_{p+1\,\ra'}(\s_1',\s_2'),\ldots,\nabla_{\bar j}q_{k\,\ra'}(\s_1',\s_2'))$ is nonzero.)
\par Denote by $((\s_1,\s_2,\ra),\bar h)$ such a maximal tuple, with $\bar h:=(h_1(\x^0,\bar z, \x^1_{[\ell]}),\ldots, h_k(\x^0,\bar z,\x^1_{[\ell]}))\in R_{N+\ell}^{<\omega}$.
\par We consider the map $\Lambda: (v_1,\ldots,v_N,v_{N+1},\ldots,v_{N+\ell})\mapsto det(\nabla_{\bar i} h_{1},\ldots, \nabla_{\bar i} h_{p},\nabla_{\bar j} h_{p+1},\ldots,\nabla_{\bar j} h_k)$. 
By construction $\Lambda(\s_1,\s_2,\ra)\neq 0$ so it doesn't vanish on a neighbourhood of
$(\s_1,\s_2,\ra)$. So, there is a neighbourhood  $U_0$ of $(\s_1,\s_2,\ra)$ where $\Lambda\restriction U_0$ is invertible. 
\par We will use the notation $(\s_1,\s_2,\ra)_{[N+\ell-k]}$ for the subtuple of $(\s_1,\s_2,\ra)$ obtained when only considering the indices in $\{1,\ldots,N+\ell\}$ which do not occur neither in $\bar i$ nor in $\bar j$.
\par We have a map $\;\hat{}: \Ge^{N+\ell}(\Se_{(\s_1,\s_2,\ra)})\to \Ge^{N+\ell-k}(\Se_{(\s_1,\s_2,\ra)_{[N+\ell-k]}})$ (applying $(\IFTA)_{e}$ to $(h_1,\ldots, h_k)$. 
Consider $R_{N+\ell}[\Lambda^{-1}]$ and let $M\subset \Ge^{N+\ell-k}(\Se_{(\s_1,\s_2,\ra)_{[N+\ell-k]}})$ be its image by the map $\hat{}$. Let $I:=\{g\in M\colon g((\s_1,\s_2,\ra)_{[N+\ell-k]})=0\}$ be an ideal in $M$. 
\par If $I=\{0\}$, then for any $f_j$ occurring in the tuple $\bar f$, since $\hat{f_j}\in I$, $f_j$ vanishes in a neighbourhood $W_j$ of $(\s_1,\s_2,\ra)$. 
So it suffices to apply the scheme $(\DL)_E$ to the pair ($\bar h_{\bar i},W)$ where $W=\bigcap W_j$, in order to find in $V(K)\cap W$ a differential tuple.
\par If $I\neq \{0\}$, let us show that we contradict the choice of $h_1,\ldots,h_k$. By $(\LFF)$, $I$ is not closed under differentiation. So there is $h\in I$ with $\nabla \hat{h}(\s_1,\s_2,\ra)_{[N+\ell-k]}\neq 0$. Since $h\in I$, there is some $s\geq 1$ such that $\Lambda^s h\in R_{N+\ell}$. Let $\tilde h:=\Lambda^s h$. Using the same argument as in \cite[Theorem 4.9]{W}, we get that $\nabla \hat{\tilde h}(\s_1,\s_2,\ra)_{[N+\ell-k]}\neq 0$ and so we can add $\tilde h$ to $h_1,\ldots,h_k$ contradicting maximality (and the fact that $(\s_1,\s_2,\ra)\notin \Ecl^{K}(K_0))$. (Indeed using Lemma \ref{hat}, we have that $\nabla h_{1}(\s_1,\s_2,\ra),\ldots, \nabla h_{k}(\s_1,\s_2,\ra),  \nabla \tilde h(\s_1,\s_2,\ra)$ are $K$-linearly independent.)
\qed

\thm \label{DLE} Let $T$ be a model-complete theory of topological $\L$-fields.
Assume that $\K\models T$ and that $\K$ satisfy $(\IFTA)_{e}$ and $(\LFF)$. Assume that the differential expansion $\K_{\delta}$ is a model of $T_\delta\cup(\DL)_{E}$. Then $\K_{\delta}$ is existentially closed in the class of  models of $T_{\delta}$. In particular if the theory $T_\delta\cup (\DL)_E$ is consistent, then it is model-complete.
\ethm 
\pr Let $\K_{\delta}\models T_{\delta}\cup (\DL)_{E}$ and suppose that $\K_{\delta}\subset \tilde \K_{\delta}$ with $\tilde \K_{\delta}\models T_\delta$. W.l.o.g. we may assume that $K$ is $\aleph_1$-saturated.
\par Let $\x=(x_1,\ldots,x_n)$ and $\xi(\x)$ be a quantifier-free $\L_{\delta}(K)$-formula of order $m$ and assume that for some tuple $\ba\in \tilde K_\delta$, $\tilde\K_\delta\models \xi(\ba)$. Since $T$ is model-complete and $\K_\delta\models T$, we may assume that we are in the case where $m\geq 1$.
We may assume that $\xi(\x)$ is of the form $\varphi(\x)\wedge\;(\bar \delta^m(\x)\in O)$, where $\varphi(\x)$ is a conjunction of $\L_\delta(K)$-equations and $O$ is an $\L(K)$-definable open subset of some cartesian product of $\tilde K$. Furthermore using equation \ref{eq:*} (and since we are quantifying existentially), we may assume $\xi(\x)$ is equivalent to a formula of the form $\varphi^{*,m}_-(\bar \delta^m(\x))\,\wedge \bar \delta^m(\x)\in O$, where $\varphi^{*,m}_-$ is a conjunction of $\L^-(K)$-equations (with possibly a smaller definable open set $O$). Denote by $V$ the $E$-variety corresponding to $\varphi^{*,m}_-(K)$.
\par Let  $K_0$ be the finitely generated subfield of coefficients occurring in the formula $\varphi^{*,m}_-$.
\par First note that if $a_1,\ldots, a_n \in \Ecl^{\tilde K}(K)$, then their successive derivatives can be expressed in terms of $a_i, \bar u_i$ and some tuples of elements of $K$, $1\leq i\leq n$ (see Notation \ref{rat}). By existentially quantifying over the elements of $\tilde K\setminus K$, we can transform the $\L_{\delta}(K)$-formula $\varphi$ into an existential $\L(K)$-formula and use the fact that $T$ is model-complete.
\par Set $\ba^i:=\delta^i(\ba)$, $0\leq i\leq m$. If all $\ba^i$, $0\leq i\leq m$, are $\Ecl$-independent over $K$, then by Corollary \ref{norel}, there are no non trivial $E$-polynomial vanishing on $(\ba^0,\ldots,\ba^m)$, so we may assume that the formula is of the form $\bar \delta^m(\x)\in O$. In that case we may conclude using the density of differential points (see Remark \ref{dense}). 
\par So from now on, let us assume this is not the case.
\par If all $\ba^i$, $0\leq i\leq m-1$, are $\Ecl$-independent over $K$, then the tuple $(\ba^i, 0\leq i\leq m-1)$, is a fortiori $\Ecl$-independent over $K_{0}$.  We first construct prepared system (with no Khovanskii formula) with a solution in $K\setminus \Ecl(K_0)$, using that $\K$ is $\aleph_1$-saturated and $T$ model-complete.
Then we apply the scheme $(\DL)_E$ and Lemma \ref{wellprep} to find a differential solution in $\K$ as close as we wish to $\bar \delta(\ba)$ (and so satisfying $\xi(\x)$).

\par So, from now on, assume furthermore there is at least a non-trivial $\Ecl$-relation over $K$ that occurs within the tuple $\bar\delta^{m-1}(\ba)$.
 Let $\ba_{[\ell]}^0=(a_1,\ldots, a_{\ell})$ be the longest sub-tuple of $\ba=(a_1,\ldots,a_n)$ which is $\Ecl$-independent over $K$ (which we may assume by re-indexing to be an initial subtuple since $\Ecl$ has the exchange property and by the remark made at the beginning such tuple is not empty). 
Then we consider the $\Ecl$-relations among $\ba^1$ over $K$ and $\ba_{[\ell]}^0$. Note that we certainly have $\Ecl$-relations among $\ba^1_{[n-\ell]}$ and $\ba^0$. Again we possibly re-index the subtuple $\ba^1_{[\ell]}$ such that these $\Ecl$-relations occur among the co-initial part of $\ba^1_{[\ell]}$. 
We rename the corresponding subtuple $\widetilde{\ba^1}$ and possibly permute the indices of $\widetilde{\ba^0}$ to match indices. 
We proceed in this way getting successively $\widetilde{\ba^2},\ldots, \widetilde{\ba^{m-1}}$.

Namely, suppose we got $\widetilde{\ba^i}$, $0\leq i<m-1$. We consider the $\Ecl$-relations among $\ba^{i+1}$ over $K$ and $\widetilde{\ba^0},\ldots \widetilde{\ba^i}$. Again we re-index  in order that the $\Ecl$-relations only occur in the co-initial part of $\ba^{i+1}$ and we rename the corresponding subtuple $\widetilde{\ba^{i+1}}$ as well as possibly permuting the indices of $\widetilde{\ba^0},\ldots, \widetilde{\ba^i}$ to match indices. Assume the length of  $\widetilde{a^i}$ is equal to $\ell_i$, $0\leq i\leq m-1$ and by the way it was constructed $n\geq \ell=\ell_0\geq \ell_{1}\geq\ldots\geq \ell_{m-1}\geq 0$.

\par For sake of simplicity let us assume that $m=1$. Let $H_1(a_{\ell+1},\bar u_{\ell+1}),\ldots,H_{n-\ell}(a_{n},\bar u_{n})$ be $n-\ell$ Khovanskii systems over $K(\ba_{[\ell]})$, setting $\ba_{[\ell]}=\ba_{[\ell]}^0$, witnessing that $a_{\ell+1},\ldots,a_{n}$ belong to $\Ecl^{L}(K(\ba_{[\ell]}))$. 
\par It implies that we can express $\delta(a_{\ell+i}),\delta(\bar u_{\ell+i})$ in terms of $a, \bar u_{\ell+i}, \delta(a_{1}),\ldots,\delta(a_{\ell})$, $1\leq i\leq n-\ell$ and finitely many elements of $K$ and their derivative occurring as coefficients of the $E$-polynomials appearing in the Khovanskii systems (see Notation \ref{rat}). Let $\bar u:=(\bar u_{\ell+1},\ldots,\bar u_n)$.
\par Let $V_H^{prep}$ be the corresponding prepared system (see Definition \ref{const}).
Since $T$ is model-complete, there exists $\bar \gamma\in O(K)$ and $\bar z\in K$ such that $\varphi_H^*(\bar \gamma,\bar z)$ holds. 
Further we may assume that $\bar \gamma\notin \Ecl^K(K_0)$ since $K$ is $\aleph_1$-saturated.
\par Then we apply Lemma \ref{wellprep} (and the scheme $(\DL)_{E}$) to get a differential solution $\bar \delta^m(\alpha)\in K$ in $V\cap O$.
 So $\K_{\delta}\models \xi(\alpha)$.\qed

\subsection{Geometric version of the scheme $(\DL)_E$.}
In this section we translate in geometric terms the scheme $(\DL)_E$. It is similar in spirit to the differential lifting scheme introduced by Pierce and Pillay, which gave another axiomatization of the class of differentially closed fields of characteristic $0$ \cite{PP}.

\

For $n\leq m\in \IN^*$, let $\pi_n^{m}:K^{m}\to K^n$ be the projection onto the first $n$ coordinates and let $\pi_{(n,n)}^{2m}:K^{m}\times K^m\to K^n\times K^n:(x,y)\mapsto (\pi_n^m(x),\pi_n^m(y))$.

\begin{definition}\label{schgeo}

Let $\K_{\delta}\models T_\delta$, 
then $\K_{\delta}$ satisfies the scheme $(\DLg)_E$ if the following holds. Let $\tilde \K$ be a $\vert K\vert^+$-saturated $\L$-elementary extension of $\K$.
Let $W:=W(\bar f)\subset \tilde \K^{2n}$ be an $E$-variety defined over $K$ and let $\bar \chi(K,\bd)$ be a neighbourhood of $0$ in $K^{2n}$ with $\bd$ in $K$. Suppose that $0\leq \dim^{\tilde \K}(\pi_n^{2n}(W)/K)=\ell<n$.  Let $\ba$ be a generic point of $\pi_n^{2n}(W)$ with $\ba_{[\ell]}$ a subtuple of $\ba$ of $\Ecl$-independent elements over $K$ and let $(\ba,\bb)$ be a generic point of $W$. Let 
$\bu_{\ell+i}$ be tuples of elements in $\tilde K$, $1\leq i\leq n-\ell$, witnessing that each component of $\ba_{[n-\ell]}$ belongs to $\Ecl^{\tilde \K}(K,\ba_{[\ell]})$.
Set $\bar u:=(\bu_{\ell+1},\ldots,\bu_n)\in \tilde K^m$ and assume that $(\ba,\bb)\in \pi_{(n,n)}^{2(n+m)}(\tau(Ann^{K[\X]^E}(\ba, \bar u))$, $\vert \X\vert=m+n$, then we can find a differential point $(\alpha,\delta(\alpha))\in W\cap K^{2n}$ with $\bar \chi((\alpha,\delta(\alpha))-(\ba,\bb),\bd)$.
\end{definition}

\par The scheme $(\DLg)_E$ as stated is not first-order. The first issue concerns expressing that a tuple is generic and the second is that a priori we have to consider all the $E$-polynomials in an annihilator. Concerning the second one, keeping the same notations as in Definition \ref{schgeo}, one only needs the $E$-polynomials in $Ann^{K[\X]^E}(\ba,\bar u)$ occurring in the Khovanskii systems used to express that each component of $\ba_{[n-\ell]}$ belongs to $\Ecl(K,\ba_{[\ell]})$.
\medskip
\par Now let us indicate why the models of $T_{\delta}$ that satisfy $(\DLg)_E$ are existentially closed (in the class of models of $T_{\delta}$) and conversely.
\par First, suppose $\K_{\delta}$ satisfies the scheme $(\DLg)_E$ and let us show that $\K_{\delta}$ satisfies $(\DL)_{E}$. So as in Definition \ref{DL}, let $H_{\bar \ell}(\x^{0},\ldots,\x^{m-1},\bar z)$ be a Khovanskii formula and let \\$\bar h(\x^0,\x^1_{[\ell_0]},\ldots, \x^{m-1}_{[\ell_{m-2}]},\bar z, \x^{m}_{[\ell_{m-1}]})$ be a $p$-tuple of $E$-polynomials with coefficients in $K$ and let $\bar i:=(i_{1},\ldots,i_{p})$, $1\leq i_{1}<\ldots<i_{p}\leq \ell_{m-1}$, be indices of variables among $\x^m_{[\ell_{m-1}]}$. Consider a neighbourhood $\bar\chi(K,\bar d)$ of $0$ and suppose
$H_{\bar \ell}(\ba^{0},\ldots,\ba^{m-1},\bar u)$ holds and that \\$(\ba^0,\ba^1_{[\ell_0]},\ldots, \ba^{m-1}_{[\ell_{m-2}]},\bar u, \ba^{m}_{[\ell_{m-1}]})\in V_{\bar i}^{reg}(\bar h)$. 
\par\noindent Set $\bar a:=(\ba^0,\ba^1,\ldots, \ba^{m-1},\bar u)$ and $\bar b:=(\ba^1,\ldots, \ba^{m},\bar w)$, where $\bar w=\bar t(\bar a)$ for $\bar t$ a tuple of $\La$-terms obtained as in Notation \ref{rat}, from Khovanskii systems occurring in $H_{\bar \ell}(\ba^{0},\ldots,\ba^{m-1},\bar u)$  witnessing the $\Ecl$-relations among the tuple $(\ba^{0},\ldots,\ba^{m-1})$. Let $n$ be the length of the tuple $(\bar a,\bar u).$

\par\noindent Let $W$ be the $E$-variety defined over $K$ by the equations: $(\x^0,\x^1_{[\ell_0]},\ldots, \x^{m-1}_{[\ell_{m-2}]},\bar z, \x^{m}_{[\ell_{m-1}]})\in V_{\bar i}^{reg}(\bar h) \wedge H_{\bar \ell}(\x^{0},\ldots,\x^{m-1},\bar z) \wedge \bar y=\bar t(\x^{0},\ldots,\x^{m-1},\bar z)$ in 
 $((\x^{0},\ldots,\x^{m-1},\bar z);(\x^{1},\ldots,\x^{m},\bar y)$.

Then proceed as in Lemma \ref{iter1} and find a generic point $\tilde \ba$ of $\pi^{2n}_{n}(W)$ in an elementary extension $\tilde \K$ of $\K$ close to $\bar a$ w.r.to $\bar\chi(\tilde K,\bar d)$  and a tuple $\tilde \bb$ (close to $\bar b$ w.r.to $\bar\chi(\tilde K,\bar d)$) such that there is a derivation in $\tilde \K$ sending $\tilde \ba$ to $\tilde \bb$. So we have that $(\tilde \ba,\tilde \bb)\in \tau(Ann^{K[\X]^E}(\tilde \ba))$. 
Then we use $(\DLg)_{E}$ to find a differential point $(\alpha,\bar \delta(\alpha))\in W(K)$ and close to $(\tilde \ba,\tilde \bb)$. 
\par Second, suppose $\K_{\delta}$ satisfies the scheme $(\DL)_E$ and consider an $E$-variety $W\subset \tilde \K^{2n}$ defined over $K$ and $\ba$ a generic point of $\pi^{2n}_{n}(W)$. Let $\bar u$ witnessing the $\Ecl$-dependence relations among the tuple $\ba$ so that a Khovanskii formula $H_{\ell}(\ba,\bar u)$ holds. By reordering the tuple $\ba$, we may assume that $a_1,\ldots,a_\ell$ are $\Ecl$-independent, $1\leq \ell\leq n-1$.
Construct the corresponding variety $W_{H_{\ell}}^{prep}$ as in Definition \ref{const}. Since $\ba$ is a generic point of $\pi^{2n}_{n}(W)$ and $\dim(\pi^{2n}_{n}(W))>0$, $\ba \not\subset \Ecl(K)$. Note that the assumption $(\ba,\bb)\in \pi_{(n,n)}^{2(n+m)}(\tau(Ann^{K[\X]^E}(\ba, \bar u))$ implies that finding a point in $W_{H_{\ell}}^{prep}$ is equivalent to find a point in $W$. 
Then using Lemma \ref{wellprep}, where we show that we can reduce ourselves to consider a well-prepared variety and how to use $(\DL)_{E}$ in order to find a differential point close to $(\ba,\bb)$.

\section{Model-complete theories of (partial) exponential fields}\label{exa}
In this section, we apply our previous results to theories of topological fields $\K$ where the topology is defined by either an ordering $<$ or by a (non-trivial) valuation map $v$. 
In the case of valued field $\K:=(K,v)$ we will replace the valuation map by a binary relation $\divi$ defined as follows:
$$v(a)\leq v(b)\;{\rm iff}\;a\;  \divi\; b.$$
Denote by $\O_{K}$ be the valuation ring of $K$ and $\M_K$ the maximal ideal of $\O_K$.  Let $D$ be a binary function symbol for division in the valuation ring $\O_{K}$, defined as follows:
$D(x,y):=
\left\{\begin{array}{ll}
\frac{x}{y}& \text{if}\; v(x)\geq v(y)\; \text{and}\; y\neq 0,\\ 
0&\;\text {otherwise,}
\end{array}\right.$\\
\subsection{The real numbers}
A. Wilkie showed that the theory of $(\bar \IR,exp)$ where $\bar \IR$ is the ordered field of real numbers is model-complete \cite[Second MainTheorem]{W}. So setting $T:=Th(\bar \IR,exp)$, Theorem \ref{ec} holds for models of $T$ since they also satisfy the implicit function theorem scheme $(\IFTA)$ and the lack of flat functions $(\LFF)$.
\subsection{The $p$-adic numbers}
Let $(\IQ_p,v)$ be the valued field of $p$-adic numbers. A. Macintyre showed that the theory of $\IQ_p$ admits quantifier elimination in the language of fields together with the binary relation symbol $\divi$ and 
for each $n\geq 2$, the predicates $P_n$ defined by $P_n(x)\;{\rm iff}\;\exists y\;y^n=x$.

Then J. Denef and L. van den Dries showed that the theory of the valuation ring $\IZ_p$ of $\IQ_p$ (or the theory of $\IQ_{p}$) enriched by all restricted power series with coefficients in $\IZ_{p}$ together with the predicates $P_n$, $n\geq 2$ and the binary function $D:\IZ_{p}^2\to \IZ_{p}$ 
for division in $\IZ_{p}$, admits quantifier elimination \cite[Theorem (1.1)]{DD}.
N. Mariaule showed that the theory of the valuation ring $\IZ_p$ of $\IQ_p$ expanded by the exponential function $E_p(x)$ (see Examples \ref{example} (5)) together with for each $n\geq 2$ the so-called decomposition functions for $E_p(x)$ is model-complete \cite[Theorem 4.4.5]{M}. We will recall below precisely what are these decomposition functions \cite[Chapter 4]{M}.

From that one can easily deduce that the theory of the partial exponential valued field $(\IQ_{p},E_p)$ is model-complete in the language of fields together with the predicates $P_n$, $n\geq 2$, the binary function $\divi$, the exponential function $E_p(x)$ and the decomposition functions. (Note that N. Mariaule proves {\it strong model-completeness} \cite[section 2 (2.2)]{D88}). So again Theorem \ref{ec} holds for $T=Th(\IQ_{p},E_p)$.
\par Now let us recall what are these decomposition functions. They are the analog of the functions {\it sin} and {\it cos} in the real case, but their definition is more complicated since $\IQ_{p}$ has infinitely many proper algebraic extensions. 

The field $\IQ_{p}$ is bounded, namely for each fixed $d\geq 2$ it has only finitely many algebraic extensions of degree $d$. So one may define a chain of finite algebraic extensions $K_{n}$ of $\IQ_{p}$ with the following properties:
\begin{enumerate}
\item $K_{n}$ contains any extension of degree $n$ of $\IQ_{p}$, 
\item $K_{n}$ is the splitting field of an irreducible polynomial $q_n\in \IQ[X]$ of degree $N_{n}$.  
\end{enumerate}

One may further assume that $q_n\in \IZ_{p}[X]$. Let $\beta_{n}$ be a root of $q_n$ and let $K_{n}=\IQ_{p}(\beta_{n})$, $\O_{K_{n}}=\IZ_{p}[\beta_{n}]$. Then $\O_{K_{n}}$ is a $\IZ_{p}$-module with basis $1,\beta_{n},\ldots,\beta_{n}^{N_{n}-1}$. Let $y\in \O_{K_{n}}$ and write it as $\sum_{i=0}^{N_{n}-1} x_{i}\beta_{n}^i$. Then $E_p(y)=\prod_{i=0}^{N_{n}-1}E_p(x_{i} \beta_{n}^i)$, with $x_{i}\in \IZ_{p}$ and 
 one adds the decomposition functions for each $E_p(x\beta_{n}^i)$, namely functions from $\IZ_{p}$ to $\IZ_{p}$ which allows to express $E_p(x\beta_{n}^i)$ in $\O_{K_n}$. Namely, write $E_p(x\beta_{n}^i)=\sum_{j=0}^{N_{n}-1} \tilde c_{i,j,n}(x) \beta_{n}^ i$.
Conversely, one has: $(\tilde c_{i,j,n}(x))_{i<N_n}=V^{-1}(E_p((\beta_n^j)^{\sigma} x))_{\sigma\in Gal(K_n/\IQ_p)}$, where $V$ is the Vandermonde matrix associated to the roots of $q_n$.
Finally since $det(V)$ might be of strictly positive valuation, one has to multiply the $\tilde c_{i,j,n}(x)$ by the norm $N_{K_n/\IQ_p}(det(V))$ in order to obtain the decomposition functions $c_{i,j,n}(x)$ \cite[page 66]{M}. 
 Let $\L_{pEC}$ be the language $\L_E$ together with  the predicates $P_n$, $n>1$, and the decomposition functions $c_{i,j,n}$, $0\leq j\leq N_n$, $i, n\in \IN^*$. Then the $\L_{pEC}$-theory $T$ of $(\IQ_p,E_p)$ is model-complete \cite[Theorem 4.4.5]{M}. 
Since $\IQ_p$ satisfies the analytic version of the implicit function theorem, we may apply Theorem \ref{ec}. Note that we made a slight formal extension of our former result since we not only use the exponential function $E_p$ but also the decomposition functions, but in view of the relationships described above between the decomposition functions and the exponential function $E_p$, there is no problem in doing so. The key point being able to transform an $\L_{pEC,\delta}$-term $t(x_1,\ldots,x_n)$ into an $\L_{pEC}$-term $t^*$ in $\bar \delta^{m_1}(x_1),\ldots,\bar \delta^{m_n}(x_n)$. 

\subsection{The completion of the algebraic closure of the $p$-adic numbers}
Let $\IC_p$ be the completion of the algebraic closure of the field $\IQ_p$ of p-adic numbers. As a valued field, $\IC_p$ is a model of the theory $\ACVF_{0,p}$ of algebraically closed valued fields of characteristic $0$ and residue characteristic $p$. It admits quantifier elimination in the language $\{+,-,\cdot,0, 1, \divi\}$ \cite{R}. (Note that A. Robinson only proved model-completeness of the theory but the quantifier elimination result is easily deduced.)  N. Mariaule showed that the theory of the valuation ring $\O_p$ of $\IC_p$ endowed with the exponential function $E_p(x)$ is model-complete \cite[Theorem 6.2.11]{M}. From that one can easily deduce that the theory $T$ of the partial exponential valued field $(\IC_{p},\divi, E_p)$ is model-complete. Since $\IC_p$ also satisfies the analytic version of the implicit function theorem, we may apply Theorem \ref{ec}. (Note that in this case since $\IC_p$ is algebraically closed, one does not need to add additional functions such as the decomposition functions).

\subsection{Non-standard extensions of $\IQ_{p}$}
Let $(K,v)$ be a valued field extending $(\IQ_p,v)$. Let $\O_{K}$ be the valuation ring of $K$ and let $\O_K\langle\xi\rangle$ be the ring of strictly convergent power series over $\O_K$ in $\xi:=(\xi_{1},\ldots,\xi_{m})$. An element $f(\xi)$ is given by $\sum_{\nu\in \IN^n} a_{\nu} \xi^{\nu}$, where $\xi^{\nu}=\xi_{1}^{\nu_{1}}\ldots\xi_{n}^{\nu_{n}}$ and  $v(a_{\nu})\mapsto +\infty$, when $\vert \nu\vert=\nu_{1}+\ldots+\nu_{n}\mapsto +\infty$. Such $f$ defines a function from ${\O_K}^n$ to $\O_K$ defined by
$f(u)=\left\{\begin{array}{ll}
\sum_{\nu\in \IN^n} a_{\nu} u^{\nu} &\mbox{for}\;\; u\in \O_{K}^n,\\
0&\mbox{otherwise}\\
\end{array}\right.$
\par The language $\L_{an}$ is the language of rings augmented by a $n$-ary function symbol for each $f\in \O_{K}\langle\xi\rangle$ and $n\geq 1$. Let $D$ be a binary function symbol for division restricted to the valuation ring as defined above.
Let $\L_{an,\divi}:=\L_{an}\cup\{{\divi}\}\cup \{P_{n}: n\geq 2\}$. and $\L_{an,\divi}^D:=\L_{an,\divi}\cup\{D\}$. Let $\K$ denote the $\L_{an,\divi}$-structure with domain $K$ and the above interpretation of the symbols of the language.
In view of the way the functions $f$ are interpreted in both $\IQ_{p}$ and $\K$, we have that $\IQ_{p}$ is an $\L_{an,\divi}$-substructure of $K$. Then using the quantifier elimination theorem of J. Denef and L. van den Dries, that if $\K$ is a model of $Th_{\L_{an,\divi}}(\IQ_{p})$, then $\K$ is an elementary $\L_{an,\divi}$-extension of $\IQ_{p}$. Now if we restrict the language $\L_{an,\divi}$ to the language $\L_{pEC}$, we get that the theory $T$ of $\K$ in this restricted language is also model-complete (and in fact equal to the theory of $(\IQ_p,E_p)$. 
In order to apply Theorem \ref{ec} to $\K_\delta$, we need to check that $\K$ satisfies $\IFTA_E^{an}$. A way to do this is to get a universal axiomatisation of $Th_{\L_{an,\divi}}(\IQ_{p})$. It will imply that any definable function from $\O_K^n$ to $\O_K$ is piecewise given by $\L_{an,\divi}$-terms and so analytic functions. (This argument was used for $\IR_{an}$ in \cite{DMM2}.)
\par We express that $K^*/(K^*)^n\cong \IQ_p^*/(\IQ_p^*)^n$ and that cosets representative of the subgroup of $n^{th}$ powers can be found in $\IN$, namely for every $x\in K^*$ there exist $\lambda, r\in \IN$ with $0\leq r<n$, $0\leq \lambda<p^{\beta(n)}$ and $\beta(n)=2v(n)+1$
and $P_n(x \lambda p^r)$ \cite[Lemma 4.2]{Be}. This can be expressed by a finite disjunction and translates the fact that $v(K^*)$ is a $\IZ$-group 
\par Then we express that $K$ is henselian in the following way. Let $p(X)\in \O_K[X]$ be an ordinary polynomial of degree $n$. Then one defines a function $h_{n}: \O_K^{n+1}\to \O_K$ sending 
$(a_0,\ldots,a_n,b)\mapsto u$ with $a_n b^n+\ldots+a_1 b+a_0=0$, $v(p(b))>0$, $v(\partial_X p(b))=0$ and $v(u-b)>0$ and to $0$ otherwise \cite[Definition 3.2.10]{CL}.

So this gives us a non-standard model of $T$ to which we may apply Theorem \ref{ec}.

\section{Construction of models of the scheme $(\DL)_E$}\label{construction}

\par In this section we will place ourselves in the same setting as in section \ref{sec:ec}.
 We show how to endow certain exponential topological fields $K$ endowed with a $V$-topology, satisfying $(\IFTA)_{e}$ and $(\LFF)$ 
 with a derivation in such a way they become a model of the scheme $(\DL)_E$.
One can follow a similar strategy as in \cite{Br}, \cite{Re} to endow certain (ordered) fields with a derivation in such a way they become a model of the scheme $(\DL)$ introduced in \cite{GP},  generalizing for certain differential topological fields the axiomatization $\CODF$ of closed ordered differential fields given by M. Singer in \cite{Si}.
\par In the proposition below, we will assume that the field $K$, as a topological space, is separable and first-countable and so its cardinality is at most $2^{\aleph_0}$. 
\begin{proposition} \label{external} 
Let $\L$ be a countable language and $\K$ be a topological $\L$-field of cardinality $\aleph_{1}$, endowed with a $V$-topology, which is definable with corresponding formula $\chi$. Assume that $K$ as topological space, is first-countable and separable. 
Suppose $\K$ satisfies $(\IFTA)_{e}$ and $(\LFF)$.
Then we can endow $K$ with a derivation $\delta$ such that $\K_{\delta}$ is a model of the scheme $(\DL)_E$.
\end{proposition}
\pr 
Let $\{\chi(K,\bar d_i):\;\bar d_i\in K, i\in \omega\}$ be a countable basis of neighbourhoods of $0$ and further assume, setting $W_i:=\chi(K,\bar d_i)$ that $W_{i+1}+W_{i+1}\subset W_{i}$. Let $D$ be a countable dense subset of $K$.
 Let $\K_0$ be the (countable, dense) $\L$-substructure of $\K$ generated by $(\bar d_i)_{i\in \omega}$ and $D$. 
 Moreover, we may assume, by Lowenheim-Skolem theorem, that $\K_{0}$ is an elementary substructure of $\K$.
Express $K$ as $K_{0}(B)$ with $B$ a subset of elements of $K$ which are $\Ecl$-independent over $K_{0}$ (so $\vert B\vert=\aleph_{1}$). Set $B:=(t_\alpha)_{\alpha<\aleph_1}$.
\cl \label{small}  
For each $W_i$, $i\in\omega$, and each $\ell\in\omega$, there are elements $s_1,\ldots, s_{\ell}\in W_i$ that are $\Ecl$-independent over $K_0$ and with the property that $s_j-t_j\in K_0$, $1\leq j\leq\ell$.
\ecl
\prc

Fix $W_i$ a neighbourhood of $0$ in $K$ and choose $t_0,\ldots,t_\ell\in B$, $\ell\in \omega$.
Since $K_0$ is dense in $K$, there are for each $0\leq j\leq\ell$, $r_{ji}\in K_0$ such that $t_j-r_{ji}\in W_{i}$. Set $s_{j}:=t_j-r_{ji}$, $0\leq j\leq \ell$. 
 The elements $s_{1},\ldots,s_{\ell}\in K$, are  $\Ecl$-independent over $K_{0}$ and belong to $W_i$.
\qed

\

We will express $\K$ as the union of an elementary chain of countable subfields $\K_{0}\preceq \K_{\alpha}$ endowed with a derivation $\delta_{\alpha}$, $\alpha<\aleph_{1}$, starting by putting on $K_{0}$ the trivial derivation $\delta_0$.
 \par By induction on $\alpha$, assume we have constructed $\K_{0}\subset \K_{\alpha}\preceq \K$ a countable $\L$-elementary substructure of $\K$ and suppose $\K_{\alpha}$ is endowed with a derivation $\delta_{\alpha}$. We want to find an $\L$-elementary extension $\K_{\alpha+1}$ of $\K_\alpha$ endowed with a derivation $\delta_{\alpha+1}$
with the following property. Given a neighbourhood of zero $W_j$ and 
a (finite) system $E$-polynomial equations $(\x^{0},\ldots,\x^{m})\in V$ with coefficients in $K_{\alpha}$, a Khovanskii formula $H_{\bar \ell}(\x^{0},\ldots,\x^{m-1},\bar z)$ and associated prepared system $(\x^0,\x^1,\ldots, \x^{m},\bar z)\in V_H^{prep}$, and tuple of $E$-polynomials $\bar h$ over $K_{\alpha}$ witnessing that
$V_H^{prep}$ is well-prepared over a finitely generated subfield of $K_{\alpha}$ (see Definition \ref{const}), if there are
$(\ba^{0},\ldots,\ba^{m},\bar b)\in V_{\bar i}^{reg}(h_1,\ldots,h_p)$ for some $\bar a:=(\ba^0,\ldots,\ba^m)$ 
and $\bar b \in K_{\alpha}$, then we can find $\beta, \bar u\in K_{\alpha+1}$ such that
$(\bar\delta_{\alpha+1}^m(\beta),\bar u)\in V_{\bar i}^{reg}(h_1,\ldots,h_p)\,\wedge\;(\bar\delta_{\alpha+1}^m(\beta),\bar u)-(\ba^0,\ldots,\ba^m,\bar b)\in W_j$.

\medskip
Let $\bar x:=(\x^0,\ldots, \x^m)$, $\x:=\x^0$, $\vert \x\vert=n$ and, keeping the notations of Definition \ref{const}, set
\begin{align*}
\Fs_{\alpha}:=&\{\big((\bar x\in V),\; H(\x^{0},\ldots,\x^{m-1},\bar z),\; (\bar x,\bar z)\in V_H^{prep},\;(h_1,\ldots,h_p)\big)\colon \\
                      &\K \models \exists \bar x\exists \bar z\;(\bar x,\bar z)\in V_{\bar i}^{reg}(h_1,\ldots,h_p) \;{\rm with}\;\vert \bar i\vert=p,\;(\bar x\in V)\; {\rm varying\;over\;all}\\       
                     &\L(K_\alpha){\rm -equations},\;H\;{\rm over\;the\;Khovanskii\;formulas \;over}\; K_{\alpha},\; (h_1,\ldots,h_p)\;{\rm a\; tuple \;in}\\
                     & K_{\alpha}[\X^m]^E\;{\rm witnessing\; that}\;V_H^{prep}\;{\rm is\; well-prepared}\}.
\end{align*}

 We will construct a differential extension $\K_{\alpha+1}$ of $\K_{\alpha}$ containing $t_\alpha$, satisfying the scheme $(\DL)_E$ relative to $\Fs_\alpha$, using Lemma \ref{iter1}. Note that Lemma \ref{iter1} had an hypothesis of saturation but it was only to ensure the existence of $\Ecl$-independent elements (over $K_{0}$). The property that these $\Ecl$-independent elements were $K_{0}$-small is replaced by finding elements $s_{1},\ldots,s_{\ell}\in W_{j}$, $\Ecl$-independent and congruent to $t_{1},\dots,t_{\ell}$ modulo $K_{0}$.
 
 \

\par Let $\big((\bar x\in V),\; H(\x^{0},\ldots,\x^{m-1},\bar z),\; (\bar x,\bar z)\in V_H^{prep},\;(h_1,\ldots,h_p)\big)\in \Fs_{\alpha},$ \\where
$(\bar x\in V)$ be a conjunction of $\L_-(K_{\alpha})$-equations, $H_{\bar \ell}(\x^{0},\ldots,\x^{m-1},\bar z)$ a Khovanskii formula, and $(h_1,\ldots,h_p)\in K_{\alpha}[\X^m]^E$ witnessing that $V_H^{prep}$ can be well-prepared.

\par Let $\ba:=(a_{1},\ldots, a_{n})$, $\bar a:=(\ba^0,\ldots,\ba^m),\;\bar u\in K_{\alpha}$ be such that $(\bar a,\bar u)\in V_{\bar i}^{reg}(h_1,\ldots,h_p)$ with $\vert \bar i\vert=p$.
Let $W_{j}$ be a fixed neighbourhood of zero. Then by Claim \ref{small} and Lemma \ref{iter1}, there is an elementary extension of $\K_\alpha$ inside $\K$, a derivation $\tilde \delta_{\alpha}$ extending $\delta_{\alpha}$ on $K_\alpha$ and a finite tuple $\bar \gamma$ of elements in $K$ such that for some $\overline{ \tilde \delta}_{\alpha}^{m}(\beta)\in \Ecl^K(K_{\alpha},\bar \gamma)$ we have $(\overline{\tilde \delta_{\alpha}}^m(\beta),\overline{\tilde u})\in V_{\bar i}^{reg}(h_1,\ldots,h_p)$
and $\overline{ \tilde \delta}_{\alpha}^{m}(\beta)-\bar a \in W_j$. 
Furthermore by Remark \ref{zero-iso}, $\Ecl^K(K_{\alpha},\bar \gamma)$ is countable.
In case $t_\alpha$ does not belong to $\Ecl^K(K_{\alpha},\bar \gamma)$, we define $\tilde \delta_\alpha(t_\alpha)=1$. Then let $K_{\alpha,1}=\Ecl(K_{\alpha}(t_\alpha, \bar \gamma)$.
\par We enumerate $\Fs_{\alpha}$ and the extension $\K_{\alpha,i}$ corresponds to where the $i^{th}$ tuple in $\Fs_{\alpha}$ has a differential solution close to the algebraic one in $W_j$. Set $\K_{\alpha}^{(1)}:=\bigcup_{i} \K_{\alpha,i}$. Then we redo the construction with $\K_{\alpha}^{(1)}$ in place of $\K_{\alpha}$ with a smaller neighbourhood of zero, say $W_{j+1}$.  Set $\K_{\alpha+1}:=\bigcup_{m} \K_{\alpha}^{(m)}$. Note that $\K_{\alpha+1}$ is countable.

\par So we described what happens at successor ordinals and at limit ordinals we simply take the union of the subfields we have constructed so far. 
Finally we express $\K$ as the union of a chain of differential subfields $\K_{\alpha}$ and given 
a (finite) system $E$-polynomial equations $(\x^{0},\ldots,\x^{m})\in V$ with coefficients in $K$, a Khovanskii formula $H_{\bar \ell}(\x^{0},\ldots,\x^{m-1},\bar z)$ and associated prepared system $(\x^0,\x^1,\ldots, \x^{m},\bar z)\in V_H^{prep}$, and tuple of $E$-polynomials $\bar h$ over $K$ witnessing that
$V_H^{prep}$ is well-prepared over a finitely generated subfield of $K$. 
We can find an index $\alpha$ such that the well-prepared system has all its coefficients in $K_{\alpha}$ and we may assume that we have a solution $(\bar a,\bar u)$ in $\K_{\alpha}$ belonging to $V_{\bar i}^{reg}(h_1,\ldots,h_p)$. Then given a neighbourhood $W_i$ of $0$, we can find $\beta, \overline{\tilde u}\in \K_{\alpha+1}$ such that 
 $(\bar \delta^m(\beta),\overline{\tilde u}\in V_{\bar i}^{reg}(h_1,\ldots,h_p)$ and each component of $(\bar \delta^m(\beta), \overline{\tilde u})$ is $W_i$-close to  $(\bar a,\bar u)$.\qed

\bigskip
\par Denote by $\L_0$ the reduct of $\L$ where we remove the exponential function and denote by $T_0$ the theory of the $\L_0$-reducts of the models of $T$. Let us assume that $T_0$ admits quantifier elimination. Then in \cite{GP}, we showed that the class of existentially closed models of $T_{0,\delta}$ was elementary, assuming that the models of $T$ satisfied Hypothesis (I). That last property is an analog for topological fields of the property of being large, as introduced by F. Pop \cite{Pop2}).
Let us first recall the following notation. 
Given a differential polynomial $p(X)\in K\{X\}$ of order $m>0$, with $\vert X\vert=1$, 
the separant $s_{p}$ of $p$ is defined as $s_{p}:=\frac{\partial}{\partial \delta^m(x)}p\in K\{X\}$.

\begin{definition}  \cite[Definition 3.5]{GP} The scheme of axioms $(\mathrm{DL})$ is the following: given a model $\K$ of $T_{0,\delta}$, $\K$ satisfies $(\mathrm{DL})$ if for every differential polynomial $p(X)\in K\{X\}$ with $\vert X\vert=1$ and $\ord_X(p)=m\geqslant 1$, for variables $\y=(y_0,\ldots,y_m)$ it holds in $\K$ that
\[
\forall z\big((\exists y(p^*(\y)=0 \land s_p^*(\y)\ne 0 )
\rightarrow  \exists x\big(p(x)=0\land s_p(x)\ne 0\wedge
\chi_\tau(\bar{\delta}^m(x)-\y, z)\big)\big).
\]
\end{definition}

By quantifying over coefficients, the axiom scheme $(\mathrm{DL})$ can be expressed in the language $\L_{-,\delta}$.

\cor Let $\K$ be a topological $\L$-field satisfying $(\IFTA)_{e}$ and $(\LFF)$ endowed with a $V$-topology which is definable with corresponding formula $\chi$. Assume that $K$ is of cardinality $\aleph_{1}$ and as a topological space, first-countable and separable.   
Then we can endow $\K$ with a derivation $\delta$ such that $\K_{\delta}$ is a model of the schemes $(\DL)_E\cup (\DL)$.
\ecor
\begin{proof}
We modify the proof of proposition above by also considering the instances of the scheme $(\DL)$ and alternating between solving a formula from scheme $(\DL)_E$ to solving a formula from scheme $(\DL)$. We observe that if $t_1,\ldots,t_n$ are $\Ecl$-independent, then they are also algebraically independent by Corollary \ref{norel}.
\end{proof}

\cor Let $\K$ be an ordered real-closed exponential field. Assume that $K$ is of cardinality $\aleph_{1}$ with a countable dense subfield.  
Then we can endow $\K$ with a derivation $\delta$ such that $\K_{\delta}$ is a model of $\CODF$ together with the scheme $(\DL)_E$.\qed
\ecor

\rem Now let us record a few cases when $K$ is separable and first-countable. 
First, suppose that $(K,v)$ is an henselian perfect valued field of equicharacteristic $0$, with value group $G$.
Denote by $\{t^g\in K\colon \;g\in G\;\&\;v(t^g)=g\}$, a family of elements of $K$ whose set of values is $G$.
Then by a result of Kaplansky, the residue field $k$ isomorphically embeds in the valuation ring of $K$ \cite[Lemma 3.8]{Fo}.  Assume that $k$ is countable and 
$\vert G\vert=\aleph_0$. 
Consider the subring of $K$ generated by $k$ and $\{t^g:\;g\in G_{>0}\}$, then it is dense in the valuation ring of $K$.
Since the inverse operation is continuous, $K$ has a dense countable subfield.
A countable basis of neighbourhoods of $0$ is given by the balls $W_g:=\{x\in K\colon v(x)>g\}$, where $g\in G_{>0}$.
\par Second, suppose now that $(K,\leq)$ is an ordered real-closed field. Either $K$ is archimedean and so it embeds into $\IR$. So assume that the archimedean valuation on $K$ is non trivial. The residue field with respect to this archimedean valuation embeds in $\IR$. In case the value group $G$ is countable, the subring of $K$ generated by $\IQ$ and $\{t^g:\;g\in G_{>0}\}$ is dense in the valuation ring of $K$.
 As a countable basis of neighbourhoods we may take 
the balls $B_{n,g}:=\{x\in K\colon \vert x \vert<\frac{1}{n}t^{g}\}$, with $n\in \IN^*$, $g\in G_{>0}$ and
where $t^{g}$ is a strictly positive element of $K$ with archimedean valuation equal to $g$.
\erem

\end{document}